\documentclass[11pt,english,twoside,a4paper]{amsart}     
\usepackage[english,french]{babel}
\usepackage{amssymb,amsmath,amsthm,stmaryrd}
\usepackage{mathrsfs}
\usepackage{graphicx,psfrag,epsfig}
\usepackage{graphicx}
\usepackage[T1]{fontenc}
\usepackage[latin1]{inputenc}
\usepackage[all]{xy}
\usepackage{mathrsfs}
\usepackage{pst-all}
\usepackage{float}
\theoremstyle{plain}

\theoremstyle{plain}

\newtheorem{thm}{Theorem}
\newtheorem{prop}{Proposition}[section]
\newtheorem{lem}{Lemma}[section]

\newtheorem{nota}{Notation}[section]

\newtheorem{cor}{Corollary}[section]
\newtheorem{pte}{Property}[section]

\theoremstyle{definition}
\newtheorem{Def}{Definition}[section]

\theoremstyle{remark}
\newtheorem{rem}{Remark}[section]
\newtheoremstyle{restate}
{\topsep}
{\topsep}
{\itshape}
{}
{\bfseries}
{.}
{ }
{\thmname{#1}\thmnote{ #3}}
\theoremstyle{restate}

\newtheorem{thm*}{Theorem}
\newtheorem{prop*}{Proposition}
\newtheorem{lem*}{Lemma}
\newtheorem{add*}{Addendum}
\newtheorem{cor*}{Corollary}
\newtheorem{pte*}{Property}
\theoremstyle{definition}
\newtheorem{Def*}{Definition}
\newtheorem{exm*}{Example}
\theoremstyle{remark}
\newtheorem{rem*}{Remark}

\theoremstyle{definition}

\theoremstyle{remark}

\newcommand{\ov}{\overline}
\newcommand{\Inf}{\mathop{\rm Inf\,}\limits}
\newcommand{\Max}{\mathop{\rm Max\,}\limits}
\newcommand{\Min}{\mathop{\rm Min\,}\limits}

\newcommand{\Dp}[2]{\frac{\partial #1}{\partial #2}}

\newcommand{\diam}{{\rm diam\,}}
\newcommand{\Argsh}{{\rm Argsh\,}}

\newcommand{\ha}{\widehat}
\newcommand{\ti}{\tilde}
\newcommand{\wti}{\widetilde}
\newcommand{\N}{\mathbb N}
\newcommand{\Z}{\mathbb Z}

\newcommand{\R}{\mathbb R}

\newcommand{\A}{\mathbb A}
\newcommand{\T}{\mathbb T}

\newcommand{\Cal}{\mathcal}

\newcommand{\mcr}{\mathscr}

\newcommand{\abs}[1]{\left\vert #1\right\vert}
\newcommand{\norm}[1]{\Vert#1\Vert}

\newcommand{\demi}{\frac{1}{2}}

\newcommand{\kp}{\frac{k}{p}}
\newcommand{\kkp}{\frac{k+1}{p}}

\newcommand{\Log}{{\rm Log\,}}
\newcommand{\setm}{\setminus}

\DeclareMathOperator{\Card}{Card}

\newcommand{\rk}{{\rm rank\,}}
\newcommand{\Aa}{\Cal{A}}
\newcommand{\Ss}{\Cal S}
\newcommand{\Tt}{\Cal T}
\newcommand{\Dd}{\Cal D}

\newcommand{\Cc}{\Cal C}

\newcommand{\Hh}{\Cal H}

\newcommand{\Rr}{\Cal R}

\newcommand{\Kk}{\Cal K}

\newcommand{\jA}{\mcr A}
\newcommand{\CC}{\mcr C}
\newcommand{\TT}{\mcr T}
\newcommand{\EE}{\mcr E}
\newcommand{\BB}{\mcr B}
\newcommand{\DD}{\mcr D}
\newcommand{\LL}{\mcr L}
\newcommand{\PP}{\mcr P}

\newcommand{\II}{\mcr I}
\newcommand{\JJ}{\mcr J}

\newcommand{\VV}{\mcr V}
\newcommand{\OO}{\mcr O}

\newcommand{\KK}{\mcr K}
\newcommand{\UU}{\mcr U}
\newcommand{\jS}{\mcr S}
\newcommand{\RR}{\mcr R}

\newcommand{\al}{\alpha}
\newcommand{\be}{\beta}
\newcommand{\ga}{\gamma}
\newcommand{\sig}{\sigma}
\newcommand{\eps}{\varepsilon}
\newcommand{\Ga}{\Gamma}
\renewcommand{\th}{\theta}
\newcommand{\vp}{\varphi}

\newcommand{\Sig}{\Sigma}
\newcommand{\De}{\Delta}
\newcommand{\om}{\omega}
\newcommand{\Om}{\Omega}
\newcommand{\lam}{\lambda}

\newcommand{\de}{\delta}

\newcommand{\Th}{\Theta}
\newcommand{\vth}{\vartheta}
\newcommand{\ze}{\zeta}

\newcommand{\ka}{\kappa}

\newcommand{\rit}{\rightarrow}
\newcommand{\ma}{\mapsto}

\newcommand{\inv}{^{-1}}
\newcommand{\lio}{_{e,\rho}}
\newcommand{\ke}{_{k,e}}
\newcommand{\je}{_{\JJ,e}}

\newcommand{\ths}{\th^*}

\newcommand{\ws}{w^*}

\newcommand{\bu}{\ov{u}}
\newcommand{\us}{u^*}
\newcommand{\hto}{{\rm{h_{top}}}}
\newcommand{\hp}{{\rm{h_{pol}}}}
\newcommand{\hw}{{\rm{h_{pol}^*}}}

 \newcommand{\8}{\scalebox{1.3}{$\infty$}}

\newcounter{paraga}[subsubsection]
\renewcommand{\theparaga}{{\bf\arabic{paraga}.}}
\newcommand{\paraga}{\medskip \addtocounter{paraga}{1} 
\noindent{\theparaga\ } }

\begin{document}
\selectlanguage{english}

\author{Clémence Labrousse, Jean-Pierre Marco}
\title[Polynomial entropies for Bott systems]{Polynomial entropies for Bott nondegenerate Hamiltonian systems.}
\thanks{Institut de Mathématiques de Jussieu, UMR 7586, {\em Analyse algébrique},
175 rue du Chevaleret, 75013 Paris.
email: labrousse@math.jussieu.fr}

\date{}

\maketitle

\begin{abstract} 
In this paper, we study the entropy of a Hamiltonian flow in restriction to an enregy level where it admits a first integral which is nondegenerate in the Bott sense. It is easy to see that for such a flow, the topological entropy vanishes. We focus on the \textit{polynomial} and the \textit{weak polynomial entropies} $\hp$ and $\hp^*$. We prove that, under conditions on the critical level of the Bott first integral and dynamical conditions on the hamiltonian function $H$,  $\hw\in \{0,1\}$ and $\hp\in \{0,1,2\}$.
\end{abstract}

\tableofcontents
\newpage

\section{Introduction}

A Hamiltonian function $H$ on a symplectic manifold $M$ of dimension $2\ell$ is \textit{integrable in the Liouville sense} when there exists a smooth map $F = (f_i )_{1\leq i\leq\ell}$ from $M$ to $\R^{\ell}$, which is a submersion on an open dense domain $O$ and whose components $f_i$ are first integrals in involution of the Hamiltonian vector field $X^H$. The fundamental example is that of systems in ``action-angle'' form on the annulus  $\A^{\ell}=\T^{\ell}\times\R^{\ell}$, that is Hamiltonian systems of the form $H(\theta,r)=h(r)$. The structure of these systems is well known: $M$ is foliated by Lagrangian tori $\T^\ell\times\{r_0\}$ on which the flow is linear with frequency $\om(r_0):=dh(r_0)$.

In the general case, the domain $O$ is a countable union of the ``action-angle domain'' where the system is symplectically conjugate to a system in action-angle form. It is well-know that the topological entropy of system in action-angle form vanishes, so the entropy of an integrable system is ``localized'' on the singular set of the map $F$. 

It has been proved by Paternain (\cite{Pat-dim4}) that if $M$ is of dimension $4$ and if the Hamiltonian vector fields possesses a first integral $f$ (independent with $H$) such that, on a regular level $\EE$ of $H$, the critical points of $f$ forms submanifold, then the topological entropy of $\phi_H$ in restriction to $\EE$ vanishes.
This include the cases when $f$ is a \emph{nondegenerate in the Bott sense} (see section 2). 

It turns out that when the topological entropy vanishes the complexity of dynamical system can be described by {\em two} non equivalent conjugacy invariants\,-\,the strong polynomial entropy $\hp$ and the weak polynomial entropy $\hw$\,-\, which depict the polynomial growth rates of the number $G_f(\varepsilon)$ (see section 3) and which satisfy $\hw\leq \hp$.

In this paper, we compute the polynomial entropies for Hamiltonian systems that are nondegenerate in the Bott sense with an additionnal dynamical condition of coherence (see section 2). In particular, we prove that, for these systems, both $\hp$ and $\hw$ take only integer values.
In section 4, we prove that $\hw(\phi)\in \{0,1\}$ and in Section 5, we prove that $\hp(\phi)\in \{0,1,2\}$.

\section{Dynamically coherent systems}

\subsection{Definition and description} 
Consider a symplectic manifold $(M,\om)$ with dimension $4$ and a smooth Hamiltonian function $H:M\rit \R$, with its associated vector field $X^H$ and its associated Hamiltonian flow $\phi_H$.
We fix a (connected component of ) a compact nondegenerate energy level $\EE$ of $H$. It is a compact connected submanifold of dimension $3$.  In  what follows, we will often consider the restrictions of the vector field and the flow to $\EE$. They are still denoted by $X^H$  and $(\phi_H)$. Since $\EE$ is compact, $\phi_H$ is complete.

A first integral $F:M\rit \R$ of the vector field $X^H$ is said to be \emph{nondegenerate in the Bott sense} on $\EE$ if the critical points of $f:=F_{|\EE}$ form nondegenerate sstrict submanifolds of $\EE$, that is the Hessian $\partial^2f$ of $f$ is non-degenerate on normal subspaces to these submanifolds. 

In the following we call  \emph{nondegenerate Bott system}  a triple $(\EE,\phi_H,f)$ where $\EE$ is a regular level energy of a Hamiltonian $H$ on $M$, and $f:=F_{|\EE}$ ther restriction to $\EE$ of a first integral $F$ of $H$ nondegenerate in the Bott sense on $\EE$.
One proves that the critical submanifolds may only be circles, Lagrangian tori or Klein bottle (see \cite{Mar-93}, \cite{F-88}).

The  critical circles for $f$ are periodic orbits of the flow $\phi_H$. Their \textit{index} is the number of negative eigenvalues of the restriction of $\partial^2f$ to a supplementary plane to $\R X^H$.

Consider a critical circle $\CC$ for $f$ such that $f(\CC)=c$. Let us summarize the two possibilities that occur (see \cite{Mar-09} for more details).
If $\CC$ has index $0$ or $2$, there exists a neighborhood $U$ of $\CC$ such that $f\inv\{c\}\cap U=\CC$ and such that the levels $f\inv\{c'\}$ for $c'$ close to $c$ are tori whose intersection with a normal plane $\Sig$ to $\CC$ are ``circles with common center $\Sig\cap \CC$''.
If $\CC$ has index $1$, there exists a neighborhood $U$ of $\CC$ such that $f\inv\{c\}\cap U$ is a stratified submanifold homeomorphic to  a ``fiber bundle'' with basis a circle and with fiber a ``cross''. 
The whole connected component $\PP$ of $f\inv\{c\}$ containing $\CC$ is a finite union of critical circles and cylinders $\T\times \R$ whose boundary is either made of one or two critical circles. 
All the critical circles contained in $\PP$ are homotopic and have index $1$. Such a stratified submanifold is called a \emph{polycycle}.
Fomenko assumes in \cite{F-88} that a polycycle contains only one critical circle, then we say that $\PP$ is a $\8$-level. A $\8$-level may be orientable or nonorientable.

Let us introduce the following terminology.

\begin{Def} Let $(\EE,\phi_H,f)$ be a nondegenerate Bott system.
The system $(\EE,\phi_H,f)$ is said \emph{dynamically coherent}  if the critical circles $\CC$ are either elliptic periodic orbits or hyperbolic periodic orbits.
\end{Def}

\begin{rem}
 If $(\EE,\phi_H,f)$ is dynamically coherent, the elliptic orbit are the circles with index $0$ or $2$ and the hyperbolic orbits are the circles with index $1$. The leaves $\T\times\R$ of the $\8$-level that contains a hyperbolic orbit are the invariant submanifolds of this orbit.
\end{rem}

\subsection{The dynamics in the neighborhood of singularities.}
We recall that a Hamiltonian system on the annulus $\T^n\times\R^n$ (with the canonical symplectic form $d\th\wedge dr$) is in \textit{action-angle form} if the Hamiltonian function $H$ only depends on the \textit{action variable} $r$. Therefore the Hamiltonian flow has the following form:
\[
\phi_t(\th,r)=(\th+t\om(r)\, [\Z^n],r),
\]
where $\om: r\ma \nabla h(r)$.

Fix a regular value  $a\in\R^2$ of  the moment map $F:=(H,f):M\rit \R^2$, and consider a compact connected component $\Tt$ of $F\inv(\{a\})$. 
The submanifold $\Tt$ is a torus and there exists a neighborhood $\ha{\Aa}$ of $\Tt$ in $M$ such that $F$ is a trivial fibration over $\ha{\Aa}$.
Then the Arnol'd-Liouville theorem shows that there exists a symplectic diffeomorphism $\Psi:\ha{\jA}\rit \T^2\times O$, $q\ma(\al_1,\al_2,I_1,I_2)$ such that  $I_1,I_2$ only depends on the value of $F$, so the Hamiltonian system associated with $H\circ\Psi\inv$ is in action-angle form.
As a consequence $\Psi\circ \phi^t_H=\psi^t\circ\Psi$ where $(\psi^t)$ is the Hamiltonian  flow on $\T^2\times O$ associated with $H\circ\Psi\inv$.
Such a domain $\ha{\Aa}$ is called a \textit{action-angle domain}.

\paraga {\it{Critical tori and Klein bottles.}}
We show that the dynamics in the neighborhood of a critical torus is the same as  the dynamics in the regular set of $f$. Moreover, up to a $2$-sheeted covering, the dynamics in the neighborhood of a Klein bottle is the same as the one in a critical torus.

\begin{prop}\label{aaTorecritique} 
Let $(\EE,\phi_H,f)$ be a nondegenerate Bott system. Let $\Tt\subset\EE$ be a critical torus of $f$. There exists a neighborhood $\ha{\Aa}$ of $\Tt$ in $M$ that is a action-angle domain.
\end{prop}

\begin{proof}
We first show hat there exists a neighborhood $U$ of $\Tt$ such that the foliation induced by $F$  is made by homotopic tori. Then we see that  the construction of the action-angle variables can be done as in the usual case by taking a family of  bases of the homology of each torus that depends smoothly on the tori (see Annex A, or \cite{Du-80}, or \cite{Au-01}).

\noindent $\bullet $ We begin by studying the foliation induced by $f$ on $\EE$. 
We endow $M$ with a Riemannian metric $g$ and we denote by $||\cdot||$ the norm associated with $g$. 
We can assume without loss of generality that $f(\Tt)=0$ and that $f$ has index $0$ on $\Tt$ (that is, the Hessian of $f$ in restriction to a transverse line to $\Tt$ is positive definite). 
Therefore, by the Morse-Bott theorem (see (\cite{BH-04}) for a recent proof), for any $q\in \Tt$, there exist a neighborhood $U_q$ of $q$ in $\EE$, a neighborhood $O_q\times I_q$ of $(0,0)\in \R^2\times\R$ and diffeomorphism $\Phi_q :U_q\rit O_q\times I_q, p\ma (\vp, \xi)$  such that $f\circ\Phi_q\inv(\vp,\xi)=\xi^2$.

Now, for $q\in \Tt$, one has the  decomposition $T_q\EE=T_q\Tt\oplus \R N$, where $N$ is the unit normal  vector to $\Tt$. Consider the vector bundle $\wti F$ over $\Tt$ with fiber $F_q:=N(q)$.
Since $\Tt$ is orientable, $\wti F$ is trivial.  Using the tubular neighborhood theorem, there exist a neighborhood $I$ of $0$ in $\R$  and a diffeomorphism $\Psi:U\rit \T^2\times I, q\ma (\th,x)$, that satisfies $\Psi(\Tt):=\{(\th,0)\,|\,\th\in\T^2\}$.

Fix $q\in \Tt$. We can assume that $U_q\subset U$.  The map 
\[
P:=\Psi\circ \Phi_q\inv:O_q\times I_q\rit \T^2\times I, (\vp,\xi)\ma (\th(\vp,\xi),x(\vp,\xi))
\]
 is a diffeomorphism on its image. Let $q'=(\vp,0)\in U_q\cap \Tt$.
The curve $\xi\ma (\th(\vp,\xi), x(\vp,\xi))$ is transverse to $\Tt$ at the point $q'$.
So $\Dp{x}{\xi}(\vp,0)\neq 0$ for all $(\vp,0)\in U_q\cap \Tt$. Assume that $\Dp{x}{\xi}> 0$ on $U_q\cap \Tt$.  
Then for $c>0$ small enough, $f\inv(\{c\})\cap U_q$ has two connected components $C^+$ and $C^-$ with $C^+\subset \{x>0\}$ and $C^-\subset \{x<0\}$. Both are graphs over $U_q\cap \Tt$.
Now $f\inv(\{c\})\cap \{x>0\}=f\inv(\{c\})\cap \{x\geq 0\}$ is closed. But its complementary $f\inv(\{c\})\cap \{x\leq 0\}$ is also closed so $f\inv(\{c\})$ is not connected. Therefore $f\inv(\{c\})\cap U$ has two connected components which are both graphs over $\Tt$.
So $U$ is foliated by homotopic tori that invariant under the Hamiltonian flow.
Fix $c_0>0$ and let us introduce the following domains:
\[
D_0^+:=\bigcup_{0\leq c\leq c_0}\{f\inv(c)\cap U\cap\{x\geq 0\}\}, \quad D_0^-:=\bigcup_{0\leq c\leq c_0}\{f\inv(c)\cap U\cap\{x\leq 0\}\}.
\]
The foliation induced by $f$ on each on these domain is trivial.

\noindent $\bullet $ We will now see that this property holds true for an energy  level $\EE'$ close to $\EE$.
Since $\EE$ is a regular level of $H$ there exists a neighborood $U$ of $\EE$ such that the Riemannian gradient $\nabla H$ does not vanish on $U$. We can assume without loss of generality that $H(\EE)=0$.

Consider the vector field $X:=\frac{\nabla H}{||\nabla H||^2}$ on $U$ and denote by $(\phi_t)$ its associated flow. Since $X$ is $C^1$, for $t$ small enough, $\phi_t$ is a diffeomorphism. Moreover, since $X.H\equiv1$, $\phi_t(\EE)\subset H\inv({t})$. 

Consider now an open neighborhood $V$ of $\Tt$ in $\EE$. We define a one-parameter family of vector fields $(Y_t)_t$ in $V$ in the following way 
\[
Y_t:=\phi_t^*\nabla (f).
\]
Then $Y_t$ depends in $C^1$ way of $t$, and we observe that $\Tt$ is a normally hyperbolic manifold for $Y_0=\nabla (f)$. 
Therefore one can apply the Hirsch-Pugh-Shub theorem of persistence of normally hyperbolic manifold \cite{HPS}: for $t$ close enough  to $0$, $Y_t$ admits a normally hyperbolic torus $\wti{\Tt}_t$ which  is $C^1$ close to $\Tt$.
We set $\Tt_t=\phi_t(\wti{\Tt}_t)$. Since $\phi_t$ is a diffeomorphism, $\Tt_t$ is a critical torus of $F$ contained in $H\inv(\{t\})$ and the  Hessian $\partial^2(F_{|_{\Tt_t}})$ of the restriction of $F$ to $\Tt_t$ has the same type than the Hessian of the restriction of $f$ to $\Tt$.
The first argument holds true and one gets two domains $D_t^+$ and $D_t^-$ as before.

\noindent $\bullet$ Consider the two domains $\ha{D}^+:=\bigcup_{t}D_t^-$ and $\ha{D}^+:=\bigcup_{t}D_t^-$. One can construct action variables in each of these two domains using the Arnol'd method ``by quadrature'' (see Annex A.2). One immediately checks that the action variables can be glued continuously along the union $\bigcup_{t}\Tt_t$ (with the convention $\Tt_0=\Tt$).
Therefore, the angle variables may be constructed by considering any Lagrangian section of the moment map $(H,F)$ as in Annex 1 step 5. 
\end{proof}

\begin{rem}\label{Klein} If $\Kk$ is a critical Klein bottle, one proves that  there exists a neighborhood $U$ of $\Kk$ in $M$ that admits a natural two-sheeted covering $\wti{U}$, such that the symplectic form $\Om$, and the functions $H$ and $F$ can be lifted to $\wti{U}$ (\cite{Z-96}).
Therefore, the study of the dynamics near a Klein bottle boils down to the study near a critical torus. Indeed,  if we denote by $\wti{\phi}_H^t$ the lifted flow and by $\pi:\wti{U}\rit U$ the canonical projection, $\phi_H^t\circ \pi=\pi\circ\wti{\phi}_H^t$.
\end{rem}

\paraga {{\it Elliptic orbits.}} Consider a critical circle $\Cc$ of $f$ (contained in $\EE$), that is, an elliptic periodic orbit. The following proposition stated in \cite{BBM-10} (and references there in) shows there exist action-angle coordinates in the neighborhood of $\Cc$.

\begin{prop}\label{aaelliptic}
There exists a neighborhood $U$ of $\Cc$ (in $M$), such that there exists canonical coordinates $(\vp,I,q,p)$ with $\{\vp,I\}=\{q,p\}=1$ such that 
\begin{itemize}
\item $H$ and  $f$ only depend on $I$ and $p^2+q^2$
\item $(\vp+2\pi,I, p, q)=(\vp,I, p, q)$
\item $\Cc$ is defined by  $I=0$ and $(p,q)=(0,0)$.
\end{itemize}
Moreover if $J:=\demi(p^2+q^2)$ one has 
\[
\Dp{H}{I}\neq 0,\quad 
\det
\begin{pmatrix}
\displaystyle\Dp{H}{I} & \displaystyle\Dp{H}{J}\\
\displaystyle\Dp{f}{I} & \displaystyle\Dp{f}{J}
\end{pmatrix}
\neq 0.
\]
\end{prop}

\begin{rem}
The open domain $U\setm \Cc$ is a action-angle domain $\ha{\jA}$.
\end{rem}

\begin{cor}\label{aaelliptic2} There exists a neighborhood $O$ of $(0,0)\in \R^2$ and  a projection 
$
\pi: \T^2\times O\rit U, (\vp,\psi,I,J)\ma (\vp, I,\sqrt{2J}\cos \psi, \sqrt{2J}\sin \psi)
$
such that the following diagram commutes 
\begin{equation*}
\xymatrix{
 \T^2\times O \ar[r]^{\phi^t} \ar[d]_{\pi}  & \T^2\times O \ar[d]^{\pi}  \\
    U \ar[r]_{\phi_H^t} & U.
  }
  \end{equation*}
  where $(\phi_t)$ is the flow associated to the vector field 
\begin{equation}\label{aaelliptic3}
  \dot{I}=\dot{J}=0,\quad \dot{\vp}=\Dp{H}{I}(I,J),\quad \dot{\psi}=\Dp{H}{J}(I,J).
\end{equation}
\end{cor}

\paraga {\it $\8$-levels and simple polycyles.} As in the case of the Klein bottles, we observe that it suffices to study orientable $\8$-levels. Indeed, one proves  that,  if $\PP$ is a nonorientable $\8$-level,  there exists a neighborhood $U$ of $\PP$ in $M$ that admits a natural two-sheeted covering $\wti{U}$, such that the symplectic form $\Om$, and the functions $H$ and $F$ can be lifted to $\wti{U}$ (see \cite{BBM-10}, or \cite{Z-96}).
Therefore,  if we denote by $\wti{\phi}_H^t$ the lifted flow and by $\pi:\wti{U}\rit U$ the canonical projection, $\phi_H^t\circ \pi=\pi\circ\wti{\phi}_H^t$.

Consider an orientable $\8$-level $\PP$ and denote by $\Cc$ the hyperbolic orbit contained in $\PP$. The following proposition is stated in \cite{BBM-10} (see also references therein).

\begin{prop}\label{coorhype}
In a neighborhood $U$ of $\Cc$ (in $M$) there exist canonical coordinates $(\vp,I,q,p)$ with $\{\vp,I\}=\{q,p\}=1$ such that 
\begin{itemize}
\item $H$ and  $F$ only depend on $I$ and $qp$,
\item $(\vp+2\pi,I, p, q)=(\vp,I, p, q)$,
\item $\Cc$ is defined by  $I=0$ and $(p,q)=(0,0)$.
\end{itemize}
Moreover, if $J=qp$, one has 
\[
\Dp{H}{I}\neq 0,\quad \Dp{H}{J}\neq 0,\quad 
\det
\begin{pmatrix}
\displaystyle\Dp{H}{I} & \displaystyle\Dp{H}{J}\\[10pt] 
\displaystyle\Dp{F}{I} & \displaystyle\Dp{F}{J}
\end{pmatrix}
\neq 0.
\]
\end{prop}

Set $O_1:(H,F)(U)\subset \R^2$ and $O_2:=(I,J)(U)\subset \R^2$. Shrinking $U$ if necessary, the map $(H,F):O_2\rit O_1$ is a diffeomorphim. We denote by $(\ha I, \ha J) :O_2\rit O_1$ its inverse. 

Let $V$ be a neighborhood of $\PP$ in $M$ that contains $U$ and set 
\[
\VV= V\bigcap H\inv(H(U))\bigcap F\inv(F(U)).
\]
We define two global coordinates $\II,\JJ$ in $\VV$ by setting 
\[
\II(z)=\ha I(H(z),F(z)),\quad \JJ(z)=\ha J(H(z),F(z)).
\]

Let us now pass to the definition of simple polycycles. Recall that a polycycle is a connected union of hyperbolic orbits $\{\Cc_1,\ldots,\Cc_q\}$ and their invariant manifolds. For any $\Cc_k$, we denote by $U_k$ the neighborhhood of $\Cc_k$ given by the previous proposition and by $(\vp_k, I_k,q_k,p_k)$ the corresponding coordinates. We set $J_k=p_kq_k$.

\begin{cor}\label{coorhype2} Fix $e\in H(U_k)$ and set $\UU\ke:=U_k\cap H\inv(\{e\})$\index{$\UU\ke$}.
\begin{enumerate}
\item the coordinates $(\vp_k,q_k,p_k)$ form a local chart of $\UU_{k,e}$, 
\item the compact set $\Cc_{k,e}:=\{(\vp_k,0,0)\,|\, \vp_k\in \T\}$ is a hyperbolic orbit.
\end{enumerate}
\end{cor}

\begin{proof} 
We denote by $D$ the open domain in $\R^3$ given by the coordinates $(I_k,q_k,p_k)$ and by $\ha{D}$ its image in $\R^2$ by the map $(I_k,q_k,p_k)=(I_k,q_kp_k)$. Let $\pi_J:(I_k,q_kp_k)\ma (q_kp_k)$. 
\vspace{0.15cm}

\noindent (1)  By the implicit function theorem, there exists a function $\II_k:\pi_J(\ha{D})\times H(U_k)\rit \R$ such that 
 \[
 H(I_k,J_k)=e \Longleftrightarrow I_k=\II_k(J_k,e), 
 \]
so  $(\vp_k,I_k,q_k, p_k)\in \UU\ke$ if and only if $I_k=\II_k(p_kq_k,e)$.
 
\noindent (2) The vector field $X^H$ restricted to  $\UU\ke$ reads:
\begin{equation*}
\begin{aligned}
\dot{\vp_k} & = \Dp{H}{I_k}(I_k,J_k)  &&= \Dp{H}{I_k}(\II(q_kp_k),q_kp_k)\\
\dot{q_k} & =-\Dp{H}{J_k}(I_k,J_k) \Dp{J_k}{q_k}(J_k) &&= -p_k\Dp{H}{J_k}(\II(q_kp_k),q_kp_k)\\
\dot{q_k} & =\Dp{H}{J_k}(I_k,J_k)\Dp{J_k}{p_k}(J_k) &&= q_k\Dp{H}{J_k}(\II(q_kp_k),q_kp_k),
\end{aligned} 
\end{equation*}
Obviously  the compact set $\{(\vp_k,0,0)\,|\, \vp_k\in \T\}$ is a hyperbolic orbit.
\end{proof}
%


We say that the polycycle $\PP$ is \emph{continuable} if there exists $\de_0>0$ such that for all $e\in\, ]e_0-\de_0,e_0+\de_0[\subset \bigcap_{1\leq k\leq q}H(U_k)$, the hyperbolic orbits $\Cc\ke$, $1\leq k\leq q$ lie in the same polycycle $\PP_e$, which is diffeomorphic to $\PP$.  
The one-parameter family $\PP_e$ is a (differentiable) deformation of $\PP$ and we set 
\[
\ha{\PP}=\bigcup_{e\in J(\de_0)}\PP_e\subset H\inv(]e_0-\de_0,e_0+\de_0[).
\]
Observe that an $\8$-level is continuable. Let us introduce the following definition.

\begin{Def}
We call \emph{simple polycycle} a continuable polycycle that satisfies the following properties:
\begin{enumerate}
\item there exist an open subset $O$ in $\R^2$,  a neighborhood $U$ of $\PP$, saturated for $F$, and a diffeomorphism 
\[
\BB: \T\times O\times ]e_0-\de_0,e_0+\de_0[\,\rit  U
\]
such that
\begin{enumerate}
\item the submanifold $V=\BB\inv(\{0\}\times O\times  ]e_0-\de_0,e_0+\de_0[)$ is transverse to $F$, 
\item $ \BB(\T\times O\times \{e\})\subset H\inv(e),\quad \forall e\in \, ]e_0-\de_0,e_0+\de_0[$,
\item $
F(\BB(\vp, x,  e))=F(\BB(\vp',x,e)),\quad \forall (\vp,\vp',x,e)\in \T^2\times O\times ]e_0-\de_0,e_0+\de_0[$,
\end{enumerate}
\item there exist two functions $\II$\index{$\II$} and $\JJ$\index{$\JJ$} in $U$, such that one can find coordinates $(\vp_k, I_k,q_k,p_k)$ in a neighborhood $U_k$ of each $\Cc_k$ such that $\II$  and $\JJ$ coincide with $I_k$ and $J_k$.
\end{enumerate}
\end{Def}

Obviously, an orientable $\8$-level is a simple polycycle.
We emphasize the following property of simple polycycles.

\begin{pte}
Let $T\ke$ be the period of the hyperbolic orbit $\Cc\ke$. Then $T\ke$ only depends on $e$, that is, all the hyperbolic orbits that lie in the same polycycle $\PP_e$ have the same period.
\end{pte}

\begin{proof}
Fix $k$ and consider the symplectic cylinder 
\[
\CC_k:=\bigcup_{e\in H(U_k)}\Cc\ke.
\]
We first observe that the restriction of $H$ to $\CC_k$ depends only on $I_k$. We set $H_{|\CC_k}(\vp_k, I_k)=h_k(I_k)$ and the vector field $X^{H_{|\CC_k}}$   reads 
\[
\dot{\vp}_k=h_k'(I_k), \quad  \dot{I}_k=0.
\]
Therefore, since $h_k$ is a diffeomorphism, $T\ke:=h'_k(\II_k(0,e))\inv$. 
Let $\chi_k:=h_k\inv$.
 Then  $T\ke:=\chi_k'(e)$.
 Now  if $k'\neq k$, $\chi_{k'}(e)=I_{k'}(0,e)=\II(0,e)=I_k(0,e)=\chi_k(e)$, that is, $\chi_k$ does not depend on $k$. We denote it by $\chi$.
Then, for any $1\leq k\leq p$, $T\ke=\chi'(e)$.
\end{proof}

\paraga {\it Maximal action-angle domains.} We denote by $\Rr(f)$ the set of regular values of $f$ and by ${\rm{Crit}}(f)$ the set of its critical values. 
If $c\in{\rm{Crit}}(f)$, we denote by $\RR_c$ the union  of the connected components of $f\inv(\{c\})$ that does not contain any critical point. We define the \emph{regular set}\index{$\RR$} of $f$ as 
\[
\RR:=f\inv(\{\Rr\})\cup\left(\bigcup_{c\in{\rm{Crit}}(f)}\RR_c\right).
\] 
We denote by $\TT_c$ the set of all critical tori of $f$ and we  introduce the  domain $\ha{\RR}:=\RR\cup\TT_c$.
By  Arnol'd-Liouville theorem and proposition \ref{aaTorecritique}, one sees that a connected component of $\ha{\RR}$ is 
 the connected component of an intersection $\Aa:=\ha{\Aa}\cap\EE$, where $\ha{\Aa}$\index{$\Aa$} is a action-angle domain. Such a domain $\Aa$  satisfies the following properties:
\begin{itemize}
\item there exist $a,b\in{\rm{Crit}}(f)$ with $a<b$ and $\Aa= f\inv(]a,b[)$,
\item for all $x\in\,]a,b[$, $f\inv(x)\cap \Aa$ is diffeomorphic to $\T^2$,
\item there is a critical point of $f$ in each connected component of $\partial \Aa$.
\end{itemize}
Therefore $\Aa$ is diffeomorphic to $\T^2\times\,]0,1[$. We say that $\Aa$ is a \emph{maximal action-angle domain of} $(\EE,\phi_H,f)$. The connected component of $\partial\Aa$ can be either an elliptic orbit, a Klein bottle or  containded in a $\8$-level.

\section{Polynomial entropies}

In this section, we briefly  define  the polynomial entropies. For a more complete introduction see \cite{Mar-09}.
\subsection{Definitions of the strong and  weak polynomial entropies.}
Given a compact metric space $(X,d)$ and a map $f:X\rit X$, for each $n\leq 1$, one can define the dynamical metric
\begin{equation}
d^{f}_n(x,y)=\max_{0\leq k\leq t-1}d(f^k(x),f^k(y)).
\end{equation}
All these metrics $d^{f}_k$ are equivalent and define the initial topology on $X$. In particular, $(X,d_k^{f})$ is compact. So for any $\varepsilon>0$, $X$ can be covered by a finite number of balls of radius $\varepsilon$ for $d^{f}_k$. Let $G_k^f(\varepsilon)$ be the minimal number of balls of such a covering.

\begin{Def}
The \emph{strong polynomial entropy} $\hp$ of $f$ is defined as
\[
\hp(f)=\sup_{\eps}\inf\left\{\sigma>0|\limsup\frac{1}{n^{\sigma}}G_n^f(\eps)\right\}=\sup_{\eps\rit 0}\limsup_{n\rit\infty}\frac{\Log G_n^f(\eps)}{\Log n}.
\]
\end{Def}

In order to introduce the weak polynomial entropy, let us set out some notations. For $\eps>0$ consider the set
\[
\BB^f_\eps:=\{B_n^f(x,\eps)\,|\,(x,n)\in X\times\N\},
\]
of all open balls of radius $\eps$ for all the distances $d_n^f$. We denote by $\CC^f(\eps)$ the set of the coverings of $X$ by balls of $\BB^f_\eps$, and by $\CC^f_{\geq N}(\eps)$ the subset of $\CC^f(\eps)$ formed by the coverings $(B_{n_i}^f(x_i,\eps))_{i\in I}$ such that $n_i\geq N$.
Given an element $C=(B_{n_i}^f(x_i,\eps))_{i\in I}$ in $\CC^f(\eps)$ and a non negative real parameter $s$, we set
\[
M(C,s)=\sum_{i\in I}\frac{1}{n_i^s}\in[0,\infty]
\]
Note that since a ball may admit several representatives of the form $B_{n_i}^f(x_i,\eps)$, $M(C,s)$ depends on the family $C$ and not only of its image. Let $N\in \N^*$. The compactness of $X$ allows us to define
\[
\De^f(\eps,s,N)=\Inf\{M(C,s)\,|\,C\in\CC^f_{\geq N}(\eps)\}\in [0,\infty].
\]
Obviously $\De^f(\eps,s,N)\leq \De^f(\eps,s,N')$ when $N'\leq N$, so one can define
\[
\De^f(\eps,s)=\lim_{N\rit\infty}\De^f(\eps,s,N)=\sup_{N\in \N^*}\De^f(\eps,s,N).
\]
The definition of the weak polynomial entropy is based on the following lemma.
\begin{lem}
There exists a unique critical value $s_c^f(\eps)$ such that
\[
\De^f(\eps,s)=0\:\:\: if\:\:\: s> s_c^f(\eps)\:\:\:\: and\:\:\:\: \De^f(\eps,s)=\infty\:\:\: if\:\:\: s< s_c^f(\eps).
\]
\end{lem}
Since $s_c^f(\eps)\leq s_c(\eps')$ when $\eps'\leq\eps$, one states the following definition.
\begin{Def}
The \emph{weak polynomial entropy} $\hp^*$ of $f$ is defined as
\[
\hp^*(f):=\lim_{\eps\rit0}s_c^f(\eps)=\sup_{\eps>0}s_c^f(\eps)\in [0,\infty].
\]
\end{Def}

The relation between the polynomial entropy and the weak polynomial entropy can be made more precise. Denote by $\CC^f_{=N}(\eps)$ the subset of $\CC^f(\eps)$ of all coverings of the form $(B_{n_i}^f(x_i,\eps))_{i\in I}$ with $n_i= N$. We set
\[
\Ga^f(\eps,s,N)=\Inf\{M(C,s)\,|\,C\in\CC^f_{= N}(\eps)\}\in [0,\infty].
\]
If\,
$
\Ga^f(\eps,s)=\limsup\limits_{N\rit\infty}\Ga^f(\eps,s,N),\,
$
one has\,\,
$
\Ga^f(\eps,s)=\limsup\limits_{N\rit\infty}\frac{1}{N}G_N^f(\eps).
$\\
As before, one checks that there exists a critical value $\ov{s}_c^f(\eps)$ such that
\[
\Ga^f(\eps,s)=0\,\,\,\textrm{if}\, s> \ov{s}_c^f(\eps)\,\,\, \textrm{and}\,\,\,\Ga^f(\eps,s)=\infty\,\,\,\textrm{if}\,\,\,s< \ov{s}_c^f(\eps).
\]
Therefore
$
\hp(f)=\lim\limits_{\eps\rit0}\ov{s}_c^f(\eps)=\sup_{\eps>0}\ov{s}_c^f(\eps).
$\\
\begin{rem} The inclusion $\CC^f_{= N}(\eps)\subset\CC^f_{\geq N}(\eps)$ yields 
$
\hp^*(f)\leq \hp(f).
$
\end{rem}
Like the topological entropy, the  polynomial entropies $\hp$ and $\hp^*$ are $C^0$-conjugacy invariant and do not depend on the choice of topologically invariant metrics on $X$.
Before giving some properties of the polynomial entropies, let us emphasize the following  important fact.

\begin{rem}
When $\hto(f)>0$, the polynomial entropies  are both infinite.
\end{rem}

\subsection{Properties of the polynomial entropies.} 
We now give without proof some basic properties of the polynomial entropies. 

\begin{pte}\label{propriete} 
Here the symbol $\ha{h}$ stands indifferently for $\hp$ or $\hp^*$.
\begin{enumerate} 
\item If $A\subset X$ is invariant under $f$, $\ha{h}(f_{\vert_A})\leq \ha{h}(f)$.
\item If $(Y,d')$ is another compact metric space and if $g :Y\rit X$ is a continuous factor of $f$, that is, if there exists a continuous surjective map $\ell:X\rit X'$ such that $\ell\circ f=g\circ \ell$, $\ha{h}(g)\leq \ha{h}(f)$.
\item If $g$  is a continous map of another compact metric space $(Y,d')$, and if $X\times Y$ is endowed with the product metric, then 
\[
\ha{h}(f\times g)=\ha{h}(f)+\ha{h}(g).
\]
\item $\ha{h}(f^m) = \ha{h}(f)$ for all $m\in\N$.
\item If $A=\cup_{i=1}^n A_i$ where $A_i$ is invariant under $f$, $\ha{h}(f_{\vert_A}) = \Max_i(\ha{h}(f_{|_{A_i}}))$.
\end{enumerate}
\end{pte}

The weak polynomial entropy satisfies moreover the following property of $\sig$-union.
\begin{prop}\textbf{The $\sig$-union property for $\hp^*$.}\label{sigUC^*} 
 If $F=\cup_{i\in\N} F_i$, where $F_i$ is closed and invariant under $f$, then
\[ 
\hp^*(f_{\vert_F}) = \sup_{i\in\N}(\hp^*(f_{|_{F_i}})).
 \]
\end{prop}

\subsubsection{Polynomial entropies for flows.} 
Let us now state briefly the definition of the polynomial entropies for flows.

For each $t\geq 1$, we denote by $\CC^\phi_{\geq t}(\eps)$ the set of coverings of $X$ of the form $C=(B_{\tau_i}(x_i,\eps))_{i\in I}$ with $\tau_i\geq t$. For such a covering $C$, we set $M(C,s)=\sum_{i\in I}\frac{1}{\tau_i^s}$ for $s\geq 0$. Finally we introduce the quantity
\[
\de^\phi(\eps,s,t)=\Inf\{M(C,s)\,|\,\in \CC^\phi_{\geq t}(\eps)\}
\]
which is monotone nondecreasing with $t$. We set $\De^\phi(\eps,s)=\lim_{t\to\infty}\de^\phi(\eps,s,t)$. As in the discrete case, one sees that there exists a unique $s_c^\phi(\eps)$ such that $\De^\phi(\eps,s)=0$ if $s>s_c^\phi(\eps)$ and $\De^\phi(\eps,s)=+\infty$ if $s<s_c^\phi(\eps)$.
The weak polynomial entropy for the continuous system $\phi$ is defined as 
\[
\hp^*(\phi)=\lim_{\eps\to 0}s_c^\phi(\eps)=\sup_{\eps>0}s_c^\phi(\eps).
\]
It turns out that $\hp^*(\phi)=\hp^*(\phi_1)$.

%

We denote by $G_t^\phi(\eps)$ the minimal number of $d_t^\phi$-balls of radius $\eps$ in a covering of $X$, and we set 
\[
\hp(\phi)=\sup_{\eps>0}\limsup_{t\to\infty}\frac{\Log G_t^\phi(\eps)}{\Log t}=\Inf\left\{\sig\geq 0\,|\, \lim_{t\to\infty}\frac{1}{t^\sig}G_t^\phi(\eps)=0\right\}.
\]
As before, one has $\hp(\phi)=\hp(\phi_1)$.

\subsection{Polynomial entropies and Hamiltonian systems.}
The polynomial entropies are particulary relevant for the study of Hamiltonian systems. The first remarkable fact is that for Hamiltonian systems in action-angle form the weak and the strong polynomial entropies do coincide as seen in the following proposition.

%

\begin{prop}\label{Csubman}  
Let $H:(\al,I)\ma h(I)$ be a $C^2$ Hamiltonian function on  $T^*\T^n$. Let $\Ss$ be a compact submanifold of $\R^n$, possibly with boundary. Then the compact $\T^n\times\Ss$ is invariant under the flow $\phi$ and one has
\[
\hp(\phi_{|_{\T^n\times\Ss}})=\hp^*(\phi_{|_{\T^n\times\Ss}})=\max_{I\in \Ss} \rk \om(I),
\]
where $\om :\Ss\rit \R^n, I\mapsto d_I(h_{|_{\Ss}})$.
\end{prop}

\begin{rem}\label{entpolpourhconv}
If $h$ is strictly convex and if $\Ss$ is a compact energy level $\Ss=h^{-1}(\{e\})$, one gets
$
\hp(\phi_{|_{\T^n\times\Ss}})=\hp^*(\phi_{|_{\T^n\times\Ss}})=n-1.
$
\end{rem}

\begin{proof} 
Recall that given a compact metric space $(X,d)$, the ball dimension $D(X)$ is by definition
\[
D(X):=\limsup_{\eps\to 0}\frac{\Log c(\eps)}{\abs{\Log\eps}}
\]
where $c(\eps)$ is the minimal cardinality of a covering of $X$ by $\eps$-balls. We will use the fact that the ball 
dimension of a compact manifold is equal to its usual dimension and that the ball dimension of the image of a compact 
manifold by $C^1$ map of rank $\ell$ is $\leq \ell$.

\vskip2mm
We endow $\R^n$ with the product metric defined by the $\max$ norm $\norm{\ }$ and the submaniflod $\Ss$ with the induced
metric. We endow the torus $\T^n$ with the quotient metric. As the pairs of points  $(\al,\al')$ of $\T^n\times\T^n$ we will have to consider are close enough to one another, we still denote by $\norm{\al-\al'}$ their distance. Finally
we endow the product $\T^n\times \Ss$ with the product metric of the previous ones.

Assume that $\rk \om=\ell$ and denote by $\Om$ the image $\om(\Ss)$. Let us denote by simply by $\phi$ the flow $\vp:=\phi_{|_{\T^n\times\Ss}}$.
\vskip2mm

We will first prove that $\hp(\vp)\leq \ell$. Let $\eps>0$. Remark that, for $N\geq 1$, if two points $(\al,I)$ and $(\al',I')$ of $\T^n\times \Ss$ satisfy 
\begin{equation}\label{eq:ball}
\norm{\al-\al'}< \frac{\eps}{2},\qquad 
\norm{\om(I)-\om(I')}\leq \frac{\eps}{2N},\qquad
\norm{I-I'}<\eps
\end{equation}
then $d_N^\vp\big( (\al,I),(\al',I'))< \eps$. Let us introduce the following coverings :

\vskip1mm $\bullet$   a minimal covering $C_{\T^n}$ of $\T^n$ by balls of radius $\eps/2$, so its cardinality
$i^*$ depends only on $\eps$;

\vskip1mm $\bullet$  a minimal covering  $(\ha B_j)_{1\leq j\leq j^*}$  of $\Ss$ by balls of radius $\eps/2$, so again 
$j^*$ depends only on $\eps$;

\vskip1mm $\bullet$ for $N\geq1$, a  minimal covering  $(\wti B_k)_{1\leq k\leq k^*_N}$ of the image $\Om$ by balls of radius
$\eps/(2N)$. 

\vskip1mm

The last two coverings form a covering $C_{\Ss}=(\ha {B}_j\cap\om\inv\wti (B_k))_{j,k}$ of $\Ss$ such that any two points $I,I'$ in the same set $\ha B_j\cap\om\inv(\wti B_k)$ satisfy the last two conditions of (\ref{eq:ball}). 
Hence we get a covering of $\T^n\times\Ss$ whose elements are contained in balls of $d_N^\vp$-radius $\eps$ by considering the products of the elements of $C_{\T^n}$ and $C_{\Ss}$.
\vskip1mm

Note that, since the ball dimension of $\Om$ is less or equal to $\ell$, given any $\ell'>\,\ell$, for $N$ large enough, $k^*_N\leq (2N/\eps)^{\ell'}$ 
Thus:
\[
G_N^\vp(\eps)
\leq i^*\,j^*\,k^*_N
\leq c(\eps) N^{\ell'}
\]
and $\ov s_c^{\vp}(\eps)\leq \ell'$. Since $\ell'>\,\ell$ is arbitrary $\hp(\vp)\leq \ell$.
Now it suffices to prove that $\hp^*(\phi)\geq \ell$. 
For $(\al,I)$ in $\T^n\times\Ss$ and $\eps>0$, we set
\begin{itemize}
\item $B(I,\eps)\subset \Ss$ the ball with respect to the induced metric of $\R^n$ on $\Ss$,
\item $B((\al,I),\eps)\subset \T^n\times\Ss$, the ball with respect to the product metric defined above,
\item $B_N^\vp((\al,I),\eps)\subset \T^n\times\Ss$ the ball with respect to the metric $d_N^\vp$ on $\T^n\times\Ss$.
\end{itemize}
Then if $(\al',I')$ in $\T^n\times\Ss$, $(\al',I')\in B_N^\vp((\al,I),\eps)$ if and only if
\[
\norm{I'-I}<\eps,\quad  \norm{k(\om(I')-\om(I))+(\al'-\al)}<\eps,\quad \forall k\in\{0,\ldots, N-1\}.
\]
Writing the various vectors in component form, one gets for $1\leq i\leq n$:
\[
\abs{\al'_i-\al_i}<\eps,\quad
\om_i(I')\in \Big]\frac{(\al_i-\al'_i)-\eps}{N-1}+\om_i(I),\frac{(\al_i-\al'_i)+\eps}{N-1}+\om_i(I)\Big[.
\]
Thus $B_N^\vp((\al,I),\eps)$ has the following fibered structure over the ball $B(\al,\eps)$:
\[
B_N((\al,I),\eps)=\bigcup_{\al'\in B(\al,\eps)}\{\al'\}\times F_{\al'},
\]
where the fiber over the point $\al'$ is the curved polytope
\[
F_{\al'}= \om\inv\Big(\prod_{1\leq i\leq n}\Big]\frac{(\al_i-\al'_i)-\eps}{N-1}
 +\om_i(I),\frac{(\al_i-\al'_i)+\eps}{N-1}+\om_i(I)\Big[\Big)\bigcap B(I,\eps).
\]

Fix now a covering $C=(B_{n_i}((\al_i,I_i),\eps))_{i\in I}$ of $\CC_{\geq N}(\T^n\times \Ss)$, and denote by $F_0^{i}$ the fiber of $\al=0$ in the ball $B_{n_i}((\al_i,I_i),\eps)$ (which may be empty).
Then the set $\{0\}\times \Ss$ is contained in the union of the fibers $F_0^i$. Let $\nu$ the Lebesgue volume of this set.

Since $\rk\om=\ell$, there exists a constant $c>0$ such that the Lebesgue volume of  the fiber $F_0^i$ satisfies
\[
{\rm Vol\,}(F_0^i)\leq c \Big(\frac{2\eps}{n_i-1}\Big)^{\ell}.
\]
The sum of the volumes of the fibers must be larger than $\nu$, so
\[
\sum_{i\in I}c \Big(\frac{2\eps}{n_i-1}\Big)^{\ell}\geq \nu
\]
Then, if $s<\ell$
\[
M(C,s)=\sum_{i\in I}\frac{1}{n_i^s}= \frac{1}{c(2\eps)^{\ell}}\sum_{i\in I}c\Big(\frac{2\eps}{n_i-1}\Big)^{\ell}
\frac{(n_i-1)^{\ell}}{n_i^s}\geq \frac{\nu}{c(2\eps)^{\ell}}\frac{1}{2^{\ell}} N^{\ell-s}
\]
and therefore
\[
\De(\T^{n}\times \Ss,\eps,s)=\lim_{N\to\infty}\de(\T^n\times  \Ss,\eps,s,N)=+\infty. 
\]
Thus $s_c^\vp(\T^n\times  \Ss,\eps)\geq \ell$, and hence $\hp^*(\vp)\geq \ell$.
\end{proof}

\section{The weak polynomial entropy $\hw$}

This section is devoted to the proof of the following theorem.

\begin{thm}\label{thmhw} 
Let $(\EE,\phi_H,f)$ be a dynamically coherent system. Then 
\[
\hw(\phi_H)\in\{0,1\}.
\]
\end{thm}

Before giving the proof we recall the notion of maps with contracting fibered structure that has been introduced in \cite{Mar-93}.

\begin{Def}\label{def:contract}
Let $(E,d_E)$, $(X,d_X)$ be compact metric spaces and consider two continuous maps  $\vp:E\to E$ and $\psi:X\to X$. 
We say that $(E,\vp)$ has a contracting fibered structure over $(X,\psi)$ when the following conditions hold true.
\vskip1mm
(i) $E$ is metrically fibered over $X$ : there exists a surjective continous map $\pi:E\to X$, a metric space $(F,d_F)$
 and a finite open covering $(U_i)_{1\leq i\leq m}$ of $X$ such that for each $i$ there exists an isometry 
$\phi_i:\pi\inv(U_i)\to U_i\times F$ (this latter space being equipped with the product metric). 
We write $\phi_i(z)=(\pi_i(z),\varpi_i(z))\in U_i\times F$.
 \vskip1mm
(ii) $(X,\psi)$ is a factor of $(E,\vp)$ relative to $\pi$: $\psi\circ \pi=\pi\circ \vp$.
 \vskip1mm
(iii) If $z,z'$ are two points of $E$ such that there exists $i$ and $j$ in $\{1,\ldots,n\}$ such that $z,z'\in\pi_i\inv(U_i)$
and $\phi(z),\phi(z')\in\pi_j\inv(U_j)$, then 
\[
d_F\big(\varpi_j(\vp(z)),\varpi_j(\vp(z'))\big)\leq d_F\big(\varpi_i(z),\varpi_i(z')\big).
\]
\end{Def}

A simple example of a map with contracting fibered structure is the one of a diffeomorphism $\vp$ of a manifold $M$ that admits a compact invariant manifold $N$ which is normally hyperbolic: then its stable manifold $W^+(N)$ admits an invariant foliation by the stable manifolds of the points of $N$, and there exists a projection $\pi$ from a neighborhood $E$ of $N$ in $W^+(N)$ to $N$ which  associate with each point $x$ the unique point $a\in N$ such that $x\in W^+(a)$. It is not dificult to see that one can choose a Riemanniann metric on $M$ and the neighborhood $E$ is such a way that $E$ is invariant under $\vp$ and $(E,\vp)$ admits a contracting fibered structure over $(N,\vp_N)$.

\begin{prop}\label{prop:contract} Let $(E,d_E)$, $(X,d_X)$ be metric spaces, and $\vp:E\to E$, $\psi:X\to X$ be continuous maps,
such that $(E,\vp)$ admits a contracting fibered structure over $(X,\psi)$.
Then 
\[
\hp(\vp)=\hp(\psi).
\]
\end{prop}

\begin{proof} We already know that $\hp(\vp)\geq \hp(\psi)$ by the factor property.
To prove the converse inequality, consider a finite open covering $(U_i)_{i\in I}$ of $X$ adapted to the fibered structure
and let $\eps_0>0$ be the Lebesgue number of this covering (so each set of diameter less than $\eps_0$ for $d_X$
is contained in one of the $U_i$). 

Let now $N\geq1$ be fixed, choose $\eps<\eps_0/2$ and consider a ball $B^X\subset X$ of  $d_N^\psi$--radius less than 
$\eps$. In particular, $B^X$ has diameter less than $\eps_0$ (for $\de$), so $B^X$ is contained in an element $U_{i_0}$ of the covering. Consider then a ball $B^F$ of radius $\eps$, in the fiber $(F,d_F)$.
As $B^X\subset U_{i_0}$, one can define the set
\[
P=\phi_{i_0}\inv \big(B^X\times B^F\big).
\]
We want to prove that $P$ has diameter less than $2\eps$ for the distance $d_N^\vp$.

For $z,z'$ in $P$, let $x=\pi_{i_0}(z)$ and $x'=\pi_{i_0}(x')$, so $x$ and $x'$ lie in $B^X$.
Note that for $0\leq k\leq N$, $\psi^k(B^X)$ has diameter less then $\eps_0$ and so is contained in some open set $U_{i_k}$ of the covering.
So, for $0\leq k\leq N-1$, the fibered structure yields the equality:
\[
d_E(\vp^k(z),\vp^k(z'))=\Max \Big(d_X\big(\psi^k(x), \psi^k(x')\big),
d_F\big(\varpi_{i_k}(\vp^k(z)), \varpi_{i_k}(\vp^k(z'))\big)\Big).
\]
Now by induction, using the inclusion $\psi^k(B^X)\subset U_{i_k}$:
\[
d_F\Big(\varpi_{i_k}(\vp^k(z)), \varpi_{i_k}(\vp^k(z'))\Big)
\leq d_F\Big(\varpi_{i_0}(z), \varpi_{i_0}(z')\Big)<2\eps
\]
and on the other hand
$
d_X\big(\psi^k(x), \psi^k(x')\big)<2\eps
$
since $x,x'\in B^X$, so 
\[
d_E(\vp^k(z),\vp^k(z'))<\eps.
\]
This proves that $P$ has $d_N^\vp$--diameter less than $2\eps$. We denote by $P(B^X,B^F)$ this set.

We now fix a minimal covering $B_1^X,\cdots,B_n^X)$ of $X$ by balls of radius $\eps$ for $d_N^\psi$, and a finite covering $B_1^F\cdots,B_m^F)$ of the fiber $F$ by balls of radius $\eps$ for $d_F$. To each pair $(B_i^X,B_j^F)$, we associate the subset $P_{ij}=P(B_i^X,B_j^F)$ of $E$. It is easy to see that $(P_{ij})_{1\leq i\leq n, 1\leq j\leq m}$ is a covering of $E$ by subsets of diameter less than $2\eps$ for $d_N^\vp$. Then
\[
G_N^{\vp}(2\eps)\leq m\cdot G_N^\psi(\eps),
\]
which yields $\hp(\vp)\leq\hp(\psi)$.
\end{proof}

We are now in position to give the proof of the theorem.

\begin{proof}[proof of theorem \ref{thmhw} ] 
We assume without loss of generality that $H(\EE)=0$.
For two values $a<b$ of $f$, we denote by $cc(f\inv(]a,b[))$ (resp $cc(f\inv(\{a\}))$) a connected component of $f\inv(]a,b[)$  (resp $f\inv(\{a\})$). Such a domain is obvioulsy invariant by $\phi_H$. 
The strategy of the proof consists in choosing a suitable finite covering of $\EE$, by such domains $D$ and to compute $\hw(\phi_H,D):=\hw((\phi_H)_{|_D})$ for each of them. 
We choose the domains such that only the following four different cases occur:
\begin{enumerate}
\item $cc(f\inv(]a,b[))$ is  a maximal action angle domain.
\item there exists $c\in\, ]a,b[$ such that $f\inv(c)\cap cc(f\inv(]a,b[))$ is  a Klein bottle, and $]a,b[\setm\{c\}\subset \Rr(f)$. 
\item $cc(f\inv(\{a\}))$ is an elliptic orbit.
\item $cc(f\inv(\{a\}))$ is a $\8$-level.
\end{enumerate}



\noindent $\bullet$\, \textbf{Cases (1) and (2):} Using remark \ref{Klein} and property \ref{propriete} (2), we see that  (2) boils down to (1).
Let $\ha{\Aa}\subset M$ be an action angle domain such that  $cc(f\inv([a,b]))\subset\Aa$. 
There exists an open domain $U\subset \R^2$ and a symplectic diffeomorphism $\Psi: \Aa\rit  \T^2\times U, m\ma (\th,r)$ such that $r=R\circ(H,f)$ where $R$ is a diffeomorphism between two open domains of $\R^2$. Moreover $\phi_H$ is conjugate to the Hamiltonian flow $\psi:=(\psi_t)$ on $\T^2\times U$ associated with $H\circ\Psi\inv$.
One has $\Psi(cc(f\inv(]a,b[)))=\T^2\times R(0,]a,b[)$. 
Consider the sequences $(a_n)_{n\in\N^*}$ and $(b_n)_{n\in\N^*}$ defined by 
\[
a_n= a +\frac{1}{n}(b-a),\:\:\:\: b_n= b -\frac{1}{n}(b-a)
\]
For $n\geq 2$, we set  $K_n:=\Psi(cc(f\inv([a_n,b_n])))=\T^2\times R(0,[a_n,b_n])$. 
So
\[
cc(f\inv(]a,b[))=\bigcup_{n\geq 2} \Psi\inv(K_n).
\]
By propositions \ref{sigUC^*}  and \ref{Csubman}, one gets
\[
\hw(\phi_H,cc(f\inv(]a,b[)))=\sup_{n\in\N}\,\hw (\psi,\T^2\times K_n)\in\{0,1\}.
\]
\vspace{0.2cm}

\noindent $\bullet$\,  \textbf{Case (3):} Let $\Cc:=cc(f\inv(\{a\}))$ be the elliptic orbit.  
The time-one map of the flow $\phi_H$ restricted to $\Cc$ is conjugate  to a rotation, so $\hw(\phi_H, \Cc)=\hp(\phi_H,\Cc)=0$.

\vspace{0.2cm}

\noindent $\bullet$\,  \textbf{Case (4):} 
Let $\PP=:=cc(f\inv(\{a\}))$ be the $\8$-level.
Let $\Cc$ be the hyperbolic contained in $\PP$ and denote by $W^s$ its stable manifold.
Then, as before,  $\hp(\phi_H,\Cc)=0$.
Now $\PP\setm\Cc$ has two connected components $W_1$ and $W_2$.
For $i=1,2$, there exists a domain $D_i\subset W_i\cup \Cc$ such that 
\[
\Cc\subset D_i\subset W^s \quad {{\rm and}} \quad W_i=\bigcup_{n\in\N}(\phi^1_H)^{-n}(D_i).
\] 
We can assume that $D_i$ is small enough so that $(D_i, (\phi_H)_{| W_i\cup\Cc})$ admits a contracting fibered structure over $(\Cc,(\phi_H)_{|\Cc})$.
Therefore $\hp(\phi_H, D_i)=\hp(\phi_H,\Cc)=0$, which yields  $\hw(\phi_H,D_i)=0$.
Applying proposition \ref{sigUC^*}, one gets 
\[
\hw(\phi_H,\bigcup_{n\in\N}(\phi^1_H)^{-n}(D_i))=0,
\]
and  $\hw(\phi_H, \PP)=0$.
\end{proof}

\section{The polynomial entropy $\hp$}

This section is devoted to the following main result of this paper. 

\begin{thm}
Let $(\EE,\phi_H,f)$ be a dynamically coherent system. Then 
\[
\hp(\phi_H)\in\{0,1,2\}.
\]
Moreover, $\hp(\phi_H)=2$ if and only if $\phi_H$ possesses a hyperbolic orbit.
\end{thm}

\subsection{Sketch of proof}
For  two values $a<b$ of $f$, we denote by $cc(f\inv([a,b]))$ an arbitrary connected component of $f\inv([a,b])$. Such a domain is obviously invariant by $\phi_H$. The strategy of the proof consists in choosing a suitable finite covering of $\EE$ by such domains and to compute $\hp(\phi_H,cc(f\inv([a,b])))$ for each of them. 
We choose the domains such that  the following four different cases only occur:
\begin{enumerate}
\item  $cc(f\inv([a,b]))$ is contained in a action-angle domain.
\item there exists $c\in \,]a,b[$ such that $f\inv(c)\cap cc(f\inv(]a,b[))$ is  a Klein bottle, and $[a,b]\setm\{c\}\subset \Rr(f)$
\item $f\inv(a)\cap cc(f\inv([a,b]))$ is an elliptic orbit and  $]a,b]\subset \Rr(f)$. 
\item $f\inv(a)\cap \ov{cc(f\inv(]a,b]))}$ is contained in an $\8$-level  and  $]a,b]\subset \Rr(f)$.
\end{enumerate}
\vspace{0.4cm}

\noindent $\bullet$ \textbf{Case (1) and (2)}. As before, we just have to study the case (1). By proposition \ref{Csubman}, one can immediately conclude that:
\[
\hp(\phi_H,cc(f\inv([a,b]))\in \{0,1\}.
\]
\vspace{0.4cm}

\noindent $\bullet$ \textbf{Case (3)}. We will use corollary \ref{aaelliptic2} and proposition \ref{Csubman}.  We denote by $\Cc$ the elliptic orbit $\Cc=f\inv(\{a\})\cap cc(f\inv([a,b]))$. We  assume that $cc(f\inv([a,b]))$ is contained in an open  neighborhood $U$ of $\Cc$ with coordinates $(\vp,I,q,p)$ as in proposition \ref{aaelliptic}. 

Let $O:=I(U)\times J(U)$ and consider the flow $\phi$ on $\T^2\times O$ associated with the vector field (\ref{aaelliptic3}), defined in corollary \ref{aaelliptic2}. 
Then, by corollary \ref{aaelliptic2}, 
\[
\hp(\phi_H, cc(f\inv([a,b])))\leq \hp(\phi^1,\pi\inv(cc(f\inv([a,b])))).
\] 
Now by proposition \ref{Csubman},
$
\hp(\phi^1,\pi\inv(cc(f\inv([a,b]))))\in \{0,1\}.
$
\vspace{0.15cm}

\noindent 1) If $\hp(\phi^1,\pi\inv(cc(f\inv([a,b]))))=0$, then $\hp(\phi_H, cc(f\inv([a,b])))=0$.
\vspace{0.2cm}

\noindent  2) If $\hp(\phi^1,\pi\inv(cc(f\inv([a,b]))))=1$ then $\hp(\phi_H,f\inv([a,b]))=1$. Indeed, show that $\hp(\phi_H,f\inv([a,b]))\geq1$.
We consider $(I,J)$ as functions on the values $e,\rho$ of $H$ and $F$ and conversely, we consider $H$ as functions on the variables $I,J$. We set 
\[
\Ss:  \{(I(0,\rho), J(0,\rho))\,|\, \rho\in [a,b]\},\quad  {{\rm and}}\quad  \om :\Ss\rit \R^n : (I,J)\mapsto d_{(I,J)}(H_{|_{\Ss}}).
\]
Since $\hp(\phi^1,\pi\inv(cc(f\inv([a,b]))))=1$, by proposition \ref{Csubman}, there exists  a value $c\in [a,b]$ such that $\rk \om(I(0,c),J(0,c))=1$. By lower semi-continuity of  the $\rk$, there exists a neighborhood $V\subset [a,b]$ of $c$ such that for each $c'\in V$, $\rk \om(I(0,c'),J(0,c'))=1$. So we can assume that $c\in \,]a,b]$. Fix $\eps>0$ such that $[c-\eps,c+\eps]\subset [a,b]$.
Then 
\[
\hp(\phi_H,f\inv([a,b]))\geq \hp(\phi_H,f\inv([c'-\eps,c'+\eps]))=1,
\] the last equality coming from proposition \ref{Csubman}.
\vspace{0.4cm}

\noindent $\bullet$ \textbf{Case (4)}.  
We first observe that, as in the case of Klein bottles, by \ref{propriete} (2), it suffices to study the case where $\PP$ is orientable. We will indeed study the more general case where $\PP$ is a simple polycycle. 

Given a simple polycycle $\PP$ contained in $f\inv(0)$, there are regular values $\pm a$ of $f$ and a neighborhood 
$\VV\subset f\inv(]-a,a[)$ of $\PP$ in $\EE$  such that each connected component of  $f\inv(]0,\pm a[)\cap \VV$  is contained 
in a maximal action-angle domain.
Given such a domain, $\DD$, then $\ov\DD\cap \PP$ is a stratified submanifold of $\EE$, 
which is the ``ordered'' union of a finite number of hyperbolic  orbits, $\Cc_1,\ldots,\Cc_n$\index{$\Cc_k$}, and cylinders $W_{k,k+1}$\index{$W_{k,k+1}$} for $1\leq k\leq n$, 
such that $W_{k,k+1}$ is one common connected component of $W^-(\Cc_k)$ and $W^+(\Cc_{k+1})$ (with the convention $n+1=1$, see \cite{Mar-09} for some details in the ordering in the planar case).

\begin{figure}[h]
\begin{center}
\begin{pspicture}(4cm,5.8cm)
\psset{xunit=.8cm,yunit=.8cm}
\rput(-.8,3){
\pscustom[fillstyle=solid,fillcolor=lightgray,linestyle=solid,linewidth=0.7pt]
{\psccurve(0,4)(1.2,3.7)(3,2.2)(4.5,2.8)(6,1.4)(8,1.5)(9.5,0)(8,-1.5)(6,-1)(4.5,-1.8)(3,-.9)(0,-2)(-1.6,-1)(-2.7,-.2)(-2.2,1)(-1.9,2.5)(-1.3,3.5)}
\pscustom[fillstyle=solid,fillcolor=white,linestyle=solid,linewidth=0.7pt]
{\pspolygon(-1.26,0.745)(-1.1,.2)(-.77,-.255)(2,0)(2,2)}
\pscustom[fillstyle=solid,fillcolor=white,linestyle=solid,linewidth=0.7pt]
{\psccurve[linewidth=.06](0,0)(1.5,2)(3,0)(4.5,-1)(6,0)(7.5,1)(9,0)(7.5,-1)(6,0)(4.5,2)(3,0)(0,-1)(-1,3)(2,2)(-2,0)}
\psline[linewidth=.06]{->}(0,4)(-.2,4)
\pscircle[fillstyle=solid,fillcolor=black](1.96,1.8){.1}
\rput(2.9,1.7){$\Cc_1=\Cc_8$}
\pscircle[fillstyle=solid,fillcolor=black](-1.26,0.745){.1}
\rput(-1.6,0.9){$\Cc_2$}
\pscircle[fillstyle=solid,fillcolor=black](-.77,-.255){.1}
\rput(-.8,-.6){$\Cc_3$}
\pscircle[fillstyle=solid,fillcolor=black](3,0){.1}
\rput(3,-.4){$\Cc_4$}
\rput(3,.4){$\Cc_7$}
\pscircle[fillstyle=solid,fillcolor=black](6,0){.1}
\rput(6,-.4){$\Cc_5$}
\rput(6,.4){$\Cc_6$}
\rput(0.5,4.3){$f\inv(\{a\})$}
\rput(6,.4){$\Cc_6$}
\rput(4.5,3.8){$\DD$}
\psline[linewidth=.01](4.5,3.5)(4.5,2.5)
}
\end{pspicture}
\end{center}
\vskip-1cm
\caption{}\label{A transverse section of a simple polycycle}
\end{figure}

Since $\PP$ is simple, by definition, one can choose $a>0$ small  enough so that for each domain $\DD$ as above, setting $\DD_a=\DD\cap f\inv(]0,\pm a])$\index{$\DD_a$} (according to the initial sign), 
there exists a homeomorphism 
\[
\chi:\ov{\DD_a} \to \T^2\times [0,1]
\]
which is smooth outside the union of the hyperbolic orbits contained in $\ov{\DD_a}\cap \PP$ and such that 
$\chi\inv(\T^2\times \{c\})$ is a Liouville torus for each $c\in\,]0,1]$. 
Such a domain $\ov{\DD_a}$ will be called a {\em partial neighborhood} for $\PP$ and such a homeomorphism $\chi$ will be called a {\em compatible} homeomorphism.

We remark that if $a$ is small enough, there exists a finite set of such partial neighborhoods $\ov{\DD_a}$ whose union cover $\PP$.


We will choose a suitable finite covering of a neighborhood of $\PP$ by partial neighborhoods and we will conjugate the flow $\phi_H$ restricted to any  of them of $\PP$ to the flow $\phi$ of a ``model'' system on $\jA:=\T^2\times[0,1]$\index{$\jA$} for which we will be able to estimate the polynomial entropy.

In section 5.2,  we define the model system $\phi$  and construct a conjugacy between $\phi$ and the restriction of $\phi_H$ to a partial neighborhood. 
In section 5.3, we show that $\hp(\phi)=2$.

\subsection{Construction of the conjugacy to a $p$-model system.}

\newcommand{\hjA}{\ha{\jA}}
Consider the compact annulus $\hjA:=\T\times[0,1]$\index{$\hjA$} with coordinates $(\th,r)$. 

If $\psi=(\psi^t)_{t\in\R}$ is a flow on $\hjA$ whose orbits are the circles $\T\times\{r\}$, we define  the {\em separation function} for two points $a=(\th,r)$ and $a'=(\th',r)$ on the same orbits as follows.
Consider a lift $\wti{\psi}:=(\ti\psi^t)$ of $\psi$ to $\R\times[0,1]$ and two lifts $\ti a, \ti a'$ of $a,a'$ located in the same fundamental domain of the covering. We set  $\wti \psi^t(\ti a)=(x(t),r)$ and
$\wti \psi^t(\ti a')=(x'(t),r)$. Then  the separation function of $a$ and $a'$ is the function $E_{a,a'}:\R\to \R$ defined by
\[
E_{a,a'}(t)=\abs{x'(t)-x(t)}.
\]
Obviously,  $E_{a,a'}$ is independent of the lifts. It is a smooth, nonnegative and  periodic function.

\begin{nota}
If $\psi=(\psi^t)_{t\in\R}$ is a flow on a set $X$, we will often write  $\psi(t,x)$ instead of $\psi^t(x)$.
\end{nota}

We define  a {\em fundamental domain} for the flow $\psi$ on $\hjA$ as a subset $\KK$ of $\hjA$ of the form $\psi([0,1],\De)=\cup_{t\in[0,1]}\psi^t(\De)$, where $\De$ is a vertical segment of equation $\th=\th_0$. 
\vspace{0.2cm}

Fix $p\in \N^*$. For $1\leq k\leq p$, we set $z_k:=(\kp,0)$ and $\OO_k:=\{|\kp-\th|<\frac{1}{8p}\}$.

\begin{Def}\label{def:modelplan}
We call {\em planar $p$--model}  on $\hjA$ any continuous flow $\psi:=(\psi_t)_{t\in\R}$ that satisfies the following conditions: 
\begin{itemize}
\item (C1) If $r>0$, the orbit of any point $(\th,r)$ is the circle $\T\times\{r\}$ and there exists $\ell>0$ such that,  for any lift $\wti \psi:=(\wti{\psi}^t)$ of  $\psi$ in $\R\times]0,1]$, and  any $(x,r)\in\,\R\times\,]0,1]$, one has 
\[
0<\frac{x(t)-x(t')}{(t-t')}\leq \ell,\quad t'\neq t \in \R,
\]
where $(x(t),r)=\wti{\psi}^t(x,r)$.

\item (C2) There exists a neighborhood $\OO_k$\index{$\OO_k$} of $z_k$ such that the restriction of $\psi$ to $\OO_k$ is a flow associated with a vector field of the form
\begin{equation}\label{eq:genform}
V_k(\th,r)=\lam_k(r) \sqrt{\left(\th-\kp\right)^2+\mu_k(r)}\Dp{}{\th}
\end{equation}
where $\lam_k$ and $\mu_k$ are positive $C^1$ functions on $[0,1]$ with $\lam_k(0)>0$, $\mu_k(0)=0$ and $\mu'_k(0)>0$.
Moreover if we set $\psi^t(\th,0)=(\th(t),0)$, then $(\th(t)-\th(t'))(t-t')>0$, for any $\th\in \T\setm\{\kp, 1\leq k\leq p\}$.
\vspace{0.1cm}

\item (C3) \textbf{Torsion condition}:  If $\wti \Psi:=(\wti{\psi}^t)_{t\in\R}$  is any lift of  $\Psi$ in $\R\times[0,1]$,  one has, for any $x\in\,\R$ and $0\leq r_1<r_2\leq 1$:
\[
x_1(t) < x_2(t)
\]
where $(x_i(t),r_i)=\wti{\psi}^t(x,r_i)$.
\vspace{0.1cm}

\item (C4) \textbf{Tameness condition}: There exists a fundamental domain $\KK$ for $\psi$  such that, given two points $a$ and $a'$ in $\KK$ on the
same orbit, there exists $t_0$ such that $E_{a,a'}(t_0)$ is maximum and 
the points $\psi^{t_0}(a)$ and $\psi^{t_0}(a')$ are located inside the domain $\KK$.
\end{itemize}
\end{Def}

Let us comment this definition. The conditions {\rm (C1)} and (C2)  say that each orbit with positive $r$ is periodic and that the points $z_k$ are the only fixed points of the flow.
The torsion condition says that  the vertical is twisted to the right by the maps $\wti \psi^t$ for $t>0$. The tameness condition is essentially a technical condition that facilitate the computation of $\hp$. Due to the torsion condition, the period $T(r)$ of the periodic orbit $\T\times\{r\}$, $r>0$, is a decreasing function of $r$, so that the minimal period $T^*$ is achieved when $r=1$.
\vspace{0.2cm}

\begin{Def}\label{def:pmodel} Let $\al$ be a $C^1$ positive function and let $\psi$ be a planar $p$-model flow on $\T\times [0,1]$.
The \emph{$p$-model system on $\jA:=\T^2\times[0,1]$  associated with $\al$ and $\psi$} is the continuous flow $\al\otimes\psi: ((\al\otimes\psi)^t)_{t\in\R}$\index{$\al\otimes\psi$} defined by
\[
(\al\otimes\psi)^t(\vp,\th,r)=(\vp+t\al(r)\, [\Z], \psi^t(\th,r)).
\]
We call \emph{minimal period} of the $p$-model system the minimal period $T^*$ of the associated planar $p$-model flow.
\end{Def}

The following proposition, stated here a little bit improperly, will be made precise and  proved in section 3.3.2.

\begin{prop}\label{hpolpmodel} With the previous assumptions, given a $p$-model system on $\jA$ with large enough minimal period, then
$\hp(\al\otimes\psi)=2$.
\end{prop}

We say that a partial neighborhood $\ov{\DD_a}$ of a simple polycycle is a {\em desingularization domain} if there exists a compatible homeomorphism $\chi :\ov{\DD_a} \to \T^2\times [0,1]$ which conjugates $\phi_H$ to the flow $\al\otimes\psi$ of a $p$-model on $\jA$, that is,
\[
\forall (t,z)\in \R\times \ov{\DD_a},\quad \chi\circ\phi_H(t,z)=\al\otimes\psi\big(t,\chi(z)\big).
\]
 We can now state our main result.
\vspace{0.4cm}

\begin{thm}
Given a simple polycycle $\PP$ of the dynamically coherent system $(\EE,\phi_H,f)$, there exists  $a>0$ small enough so that any partial neighborhood $\ov{\DD_a}$ is a  desingularization domain.
\end{thm}

\begin{Def} Given a partial neighborhood $\ov{\DD_a}$ for $\PP$ and a positive function $\tau:[0,a]\to\R$,  we call \emph {proper section associated with} $\tau$ a $2$--dimensional $C^1$ submanifold $S$ of $\ov{\DD_a}$ such that for each $z\in \ov{\DD_a}\setm S$, there exists 
a unique pair $(t_z^-,t_z^+)\in\R^{*-}\times\R^{*+}$ with $t_z^+-t_z^-=\tau(f(z))$ and $\Phi_H(]t_z^-,t_z^+[,z)\cap\Sig=\emptyset$,
and such that the Poincar\'e map defined for each $z\in S$ by
\[
\wp(z)=\Phi_H\big(\tau(f(z)),z\big)
\]
is a  homeomorphism of $S$. By $C^1$ submanifold we mean here a $C^0$ submanifold whose intersection with  $\DD_a$ is $C^1$.  
\end{Def}

Given a planar $p$--model $\big(\T\times[0,1],\psi\big)$ and a continuous positive function $\al:[0,1]\to\R$,  the associated
time-$\al$ map is the map $\psi^\al: \T\times[0,1]\to \T\times[0,1]$\index{$\psi^\al$} such that
\[
\psi^\al(\th,r)=\psi\big(\al(r),(\th,r)\big).
\]
The proof of Theorem 3 will rely on the following two lemmas.
\vspace{0.3cm}

\noindent\textbf{Lemma I.} \textit{ Let $T$ be the common period of the hyperbolic orbits contained in $\PP$.
Let $\ov{\DD_a}$ be a partial neighborhood of $\PP$. Then, there exists a proper section in $\DD_a$ associated with a $C^1$ function $\tau :[0,a]\rit \R_+^*$   such that $\lim_{\rho\rit 0}\tau(\rho)=T$.}
\vspace{0.3cm}

\noindent\textbf{Lemma II.}
\textit{ Let $\ov{\DD_a}$ be a partial neighborhood of $\PP$ and let $S$ be a proper section associated with $\tau$. 
Set $\al:[0,1]\to\R: r\ma \tau(ar)$.
Then the return map $\wp$\index{$\wp$} of $S$ is $C^0$-conjugated to the time-$\al$ map of a planar $p$-model 
$(\T\times[0,1],\psi)$}.
\vspace{0.2cm}


\subsubsection{Normal coordinates in the neighborhood of hyperbolic orbits}

Fix a hyperbolic orbit $\Cc_k$ in $\ov{\DD_a}$. Set $\UU_k: \UU_{k,e_0}=U_k\cap\EE$\index{$\UU_k$} as defined in  corollary \ref{coorhype2} with coordinates  $(\vp_k, q_k,p_k)$. The vector field $X^H$ reads
\begin{equation*}
\begin{aligned}
\dot{\vp_k} & = \Dp{H}{I_k}(I_k,J_k)  &&= \Dp{H}{I_k}(\II(q_kp_k),q_kp_k)\\
\dot{q_k} & =-\Dp{H}{J_k}(I_k,J_k)\Dp{J_k}{q_k}(J_k) &&= -p_k\Dp{H}{J_k}(\II(q_kp_k),q_kp_k)\\
\dot{q_k} & =\Dp{H}{J}(I_k,J_k)\Dp{J_k}{p_k}(J_k) &&= q_k\Dp{H}{J_k}(\II(q_kp_k),q_kp_k)
\end{aligned} 
\end{equation*}
Let $\bar{f}: J_k\ma F(\II_k(J_k,e_0), J_k)$%
\[
\partial^2f=
\begin{vmatrix}
q^2\bar{f}'(q_kp_k) & q_kp_k\bar{f}''(q_kp_k)\\[8pt] 
p^2\bar{f}'(q_kp_k)  & p_kq_k\bar{f}''(q_kp_k)
\end{vmatrix}.
\]
Since the surface $\UU_k\cap\{(0,q_k,p_k)\,|\, (q_k,p_k)\in \DD\}$ is transverse to $\Cc_k$ in $\EE$, the determinant above does not vanish and $\bar{f}'(0)\neq 0$.

Denoting by $\bar{f}\inv$ the inverse of $\bar{f} :J_k\ma f(I_k(J_k,e_0), J_k)$, we set 
\[
\om_k: J_k\ma \Dp{H}{I}(\II(\bar{f}\inv(J_k)),\bar{f}\inv(J_k)),\quad \lam_k: J_k \ma \Dp{H}{J}(\II( \bar{f}\inv(J_k)),\bar{f}\inv(J_k)).
\]
One easily checks that these functions are $C^1$. 

Permutating $p_k$ and $-p_k$  and $q_k$ and $-q_k$ if necessary, we can assume that $\ov{\DD_a}\cap \UU_k$ is defined by $p_k\geq 0, q_k\geq 0$, for any $1\leq k \le p$ and that $\lam_k>0$.
Therefore, $\ov{\DD_a}\cap \UU_k$ can be parametrized by $(\vp_k, u=q_k-p_k, \rho=J_k(p_kq_k))$. 

Since $q_k+p_k=\sqrt{(q_k-p_k)^2+4q_kp_k}=\sqrt{u^2+4J_k\inv(\rho)}$ and setting $\mu_k(\rho)=4J_k\inv(\rho)$, the vector field reads
\begin{equation}\tag{$\star$}
\dot{\vp}_k=\om_k(\rho), \quad \dot{u}:=\lam_k(\rho)\sqrt{u^2+\mu_k(\rho)}, \quad \dot{\rho}=0. 
\end{equation}
In the following, we denote by $\UU_k$ the domain contained in $\UU_k\cap \ov{\DD_a}$ defined by $|u|\leq \ov{u}$, for $\ov{u}>0$ small enough independent of $k$.

\subsubsection{Construction of the proper section: proof of Lemma I}

This section is devoted to the proof of lemma I.
Since  a partial neighborhood is the closure of an action-angle domain, we will have to study the proper sections for action-angle systems. Since action-angle systems admit a foliation by invariant Kronecker tori, we begin
by studying the proper sections (suitably defined) for Kronecker flows on $\T^2$.

\paraga {\it Proper sections for minimal Kronecker flows.}
Consider a constant vector field $X=(x_1,x_2)$ with $x_2\neq 0$ on $\T^2$.
We set $\Phi: \R\times\T^2\rit \T^2: (t,\th)\ma \phi^t(\th)$. We denote by $\pi$ the canonical projection $\R^2\rit \T^2$.

\begin{Def}
A closed curve $S\subset \T^2$ transverse to $X$ is a \emph{proper section} if 
\begin{itemize}
\item for all $\th\in \T^2\setm S$, there exists $(t_\th^-,t_\th^+)\in \R_-^*\times \R_+^*$ such that
$\phi^{t_\th^+}(\th)\in S$,  $\phi^{t_\th^-}(\th)\in S$ and $\Phi(]t_\th^-,t_\th^+[,\th)\cap S=\emptyset$,
\item the number $\tau:=t_\th^+-t_\th^-$ is independent of $\th$,
\end{itemize}
We say that $\tau$ is the \emph{transition  time} associated with $S$.
\end{Def}

We denote by $\Dd_r$ the set of vector lines in $\R^2$ with rational slope. For $(q,p)\in \Z\times \N$, we denote by $D_{q,p}$ the vector line with direction vector $(q,p)$. Let $\jS$ be the subset of $\Z\times \N$ defined by $(q,p)\in \jS$ if $|q|\wedge p=1$. Obviously, the map $\Dd_r \rit \jS: D_{q,p}\ma (q,p)$ is bijective.

\begin{prop}\label{SP_Minimal}
For any $(q,p)\in \jS$ such that $X\notin D_{q,p}$, the  projection $\pi(D_{q,p})$ is a proper section.
\end{prop}

\begin{proof}
Fix $(q,p)\in \jS$ and let us study the dynamics in the lift $\R^2$ of $\T^2$. 
One has:
\begin{equation*}
\pi\inv(\pi(D_{q,p})) =\bigcup_{(m,n)\in \Z^2}\left((m,n)+D_{q,p}\right) =\bigcup_{n\in \Z}\left((0, \frac{n}{q})+ D_{q,p}\right)
\end{equation*}
the last equality coming from $\Z+\frac{p}{q}\Z=\frac{1}{q}\Z$ since $(q,p)\in \jS$.
For $n\in \Z$, we denote by $D_n$ the affine line $(0, \frac{n}{q})+ D_{q,p}$. 
Let $z\in \pi\inv(\pi(D_{q,p}))$. Let $m\in \Z$ such that $z\in D_m$. 
The time $\tau$ needed to come back to $\pi\inv(\pi(D_{q,p}))$ following the orbit $z+\R X$ is the time needed to cut the line $D_{m+1}$ or the line $D_{m-1}$. This time is independent of the choice of $z$ on $D_m$ and of the choice of $m\in \Z$.
A simple computation yields $\tau=|x_2q-px_1|\inv$.
From this, one immediately deduces that for any $z\in \R^2$, $t_z^+-t_z^-=\tau$.
\end{proof}

\paraga{\it Proper sections for action-angle systems on $\T^2\times\R^2$.}
Let $O$ be an open domain in $\R^2$ and let $h:O\rit \R$ be a $C^2$ function. 
Consider the Hamiltonian system $X^H$ on $\T^2\times O$ defined by $H(\th,r)=h(r)$. We denote by $\phi_H$ its associated flow.
Fix a regular value $e$ of $h$ and set $\Hh_e:=h\inv(\{e\})$. For $r\in O$, we set $\Tt_r:=\T^2\times\{r\}$. The torus $\Tt_r$ is   $\phi_H$-invariant
and $X^H$ is constant on $\Tt_r$, so it can be canonically identified with an element of $\R^2$.  We denote by $\pi_r:\R^2\rit \Tt_r$  the canonical projection.

\begin{Def} Let $D\in \Dd_r$ and assume that $X^H(r)\notin D$, $\forall r\in O$.

1) Let $\Th:O\rit \T^2$ be a smooth map . Set $\LL:=\{(\Th(r),r)\,|\,r\in O\}$. The \emph{proper section associated with $\LL$ and $D$} is the submanifold: 
\[
\ha{S}:=\bigcup_{r\in O}\left( \Th(r) +\pi_r(D)\right).
\]

2) Let $\Th_e:\Hh_e\rit \T^2$ be a smooth map and set  $\LL_e:=\{(\Th_e(r),r)\,|\,r\in \Hh_e\}$. The \emph{proper section associated with $\LL_e$ and $D$} is the submanifold:
\[
S:=\bigcup_{r\in \Hh_e}\left( \Th_e(r) +\pi_r(D)\right).
\]
%
\end{Def}

\begin{rem}
Fix $r$ in $\Hh_e$. For any $\th\in \Tt_r$, the circle $\Th_e(r) +\pi_r(D)$ is a proper section for the Kronecker flow induced by $\phi_H$ on $\Tt_r$. We denote by  $\tau(r)$ its associated transition time. Obviously, the function $\tau :r\ma \tau(r)$ is smooth. We say that $\tau$ is the \emph{transition function} associated with $S$.
\end{rem}

\paraga{\it Proper sections in a partial neighborhood of a simple polycycle.} Now we go back to our Bott system and our simple polycyle $\PP$,
with its neighborhood $U$ endowed with globally defined functions $\II$ and $\JJ$.
Observe that
\begin{equation}\label{actionI}
\II(z)=\int_{C(z)}\lam,
\end{equation}
where $C(z)$ is any circle $\Cc\ke(\rho)$ such that $z\in \Tt\lio$. This function is well defined. Obviously, $\II$ only depends on the values $e$ and 
$\rho$ of $H$ and $F$.  By construction its vector field $X^{\II}$ is $1$-periodic and the critical circles $\CC\ke$ are orbits of its flow $(\phi_{\II}^t)$.

Consider the partial neighborhood  $\VV\subset\wti \UU$ of $\PP$ in $M$ that contains $\ov{\DD}_a$, that is, $\VV\cap U_k=\wti\UU\cap\{p_k q_k\geq 0\}$ (for a suitable compatibe choice of the variables $p_k,q_k$).
Assume that we got another function $A$ defined on $\VV$ such that:
\begin{itemize}
\item $A$ only depends on the values $e$ and $\rho$ of $H$ and $F$, 
\item $A$ is independent of $\II$,
\item $A$ generates a $1$-periodic flow $(\phi_A^t)$. 
\end{itemize}
Then, the pair $(\II,A)$ is  a pair of \emph{action variables} as defined in \cite{Du-80}  and we can construct a symplectic diffeomorphism $\Psi : \overset{\circ}{\VV}\rit \T^2\times B :z\ma (\th_i,\th_a, \II, A)$,  where $\overset{\circ}{\VV}=\VV\cap\{p_kq_k>0\}$ and $B$ is an open domain in $\R^2$ (see Appendix A for the construction in a general case).
As a consequence, if $\wti H:=H\circ \Psi\inv$ and if $(\ti\phi^t)$ is the Hamiltonian flow associated with $\wti H$ in $\T^2\times B$, then $\Psi\circ \phi_H^t=\ti\phi^t\circ \Psi$, for all $t\in \R$.

\begin{prop}\label{sectionsur Da}
Fix $\us>0$ and set $\ze_1:=\{0\}\times\{-\us\}\times [0,a]\subset \UU_1$. Assume that $X^H$ is transverse to $\phi_A([0,T], \ze_1)$.
Then $\phi_A([0,T], \ze_1)$ is a proper section for $\phi_H$ in $\DD_a$ associated with a $C^1$ function $\tau:\,]0,a]\rit \R_+^*$.
\end{prop}

\begin{proof}
We begin with showing that $\Psi(\phi_A([0,1], \ze_1))$ is a proper section for $\ti\phi$ in $\Psi(\DD_a)$. 
The restrictions on $\DD_a$ of $\II$ and $A$ only depend on the values $\rho$ of  $F$ and we write $\II(\rho),\, A(\rho)$.
Since $\ze_1$ is only parametrized by $\rho$, $\Psi(\ze_1)$ as the following graph form:
\[
\Psi(D)=(\th_i(\II(\rho), A(\rho)), \th_a(\II(\rho), A(\rho)), \II(\rho), A(\rho)).
\]
Set $B_{e_0}:=\{(\II(\rho), A(\rho)))\,|\,\rho\in\, ]0,a]\}$.
Consider $A$ as a function on $\T^2\times B_{e_0}$ and let $\ti\phi_A$ be the Hamiltonian flow associated with $A$  in $\T^2\times B_{e_0}$.
Then
\begin{multline}
\ti\phi_A([0,T],  (\th_i(\II, A), \th_a(\II, A),  \II, A))\\
 =\bigcup_{\II, A}\left\{(\th_i(\II, A), \th_a(\II, A))+\pi_{\II, A}(D_{0,1})\right\}\times\{\II, A\}.
\end{multline}
Let us denote by $\ti\tau$ the transition function associated with $\Psi(\phi_A([0,1], \ze_1)$.
Obviously, since $\Psi$ conjugates the flows $\phi_H$ and $\ti\phi$, $\phi_A([0,1], \ze_1)$ is a proper section for $\phi_H$ with well defined transition function 
$\tau(f(z))=\ti\tau(\II(f(z)), A(f(z)))$, thanks to the transversality assumption . 
\end{proof}

We will now construct such an action variable $A$. Then we show that the proper section got in the previous proposition has a well defined continuation to the polycycle $\PP$, with a $C^1$ transition function $\tau$.
\vspace{0.2cm}

\noindent \textbf{Construction of $A$.} 
The construction of $A$ is based on the following lemma. 

\begin{lem}
There exists a $3$-dimensional submanifold $\Pi$ in $\VV$ which is transverse to $\ha\PP:=\bigcup_{e\in [e_0-\de_0,e_0+\de_0]}\PP_e$ and such that for any $1\leq k \leq p$, $\Pi\cap U_k:=\{\vp_k=0\}$.
\end{lem}

\begin{proof}
Consider the submanifold $\ha\Pi:=\BB(\{0\}\times O\times ]e_0-\de_0,e_0+\de_0[)$ (see the definition of a simple polycycle). 
Let $k\in \{1\dots,p\}$. Since $\ha\Pi$ is transverse to $F$, $\ha\Pi$ is transverse to $\Cc_k$. We set $(\vp_k^0, I_k(e_0,0),0,0):=\ha\Pi\cap \Cc_k$.
Consider the symplectic diffeomorphism $(\vp_k, I_k,q_k,p_k)\ma (\vp_k-\vp_k^0, I_k,q_k,p_k)$ in $U_k$. We still denote by $\vp_k$ the first variable. In these new coordinates, $\Pi\cap \Cc_k:=(0, I_k(e_0,0),0,0)$. By transversality, there exists a neighborhood $V_k\subset U_k$ in which $\ha\Pi$ has the following graph form:
\[
\ha\Pi\cap V_k:=\{(\vp_k(I_k,p_k,q_k),I_k,p_k,q_k)\}.
\]
Consider the symplectic diffeomorphism 
\begin{equation}\label{coordu_kv_k}
(\vp_k, I_k,q_k,p_k)\ma (\vp_k, u_k:=\frac{1}{\sqrt{2}}(q_k-p_k),v_k:=\frac{1}{\sqrt{2}}(q_k+p_k)).
\end{equation}
Fix $\bu>0$ such that $\{(\vp_k,I_k,u_k,v_k)\in U_k\,|\,|u_k|\leq \bu\}\subset V_k$.
Let $\eta_k$ be a bump function  on $U_k$  with support in the domain $\{u\leq \bu\}\subset V_k$
and consider the submanifold $\Pi$ defined by
\[
\Pi\cap U_k:=\{(1-\eta_k)\vp_k(I_k,p_k,q_k),I_k,p_k,q_k)\},\quad 1\leq k\leq p,
\] 
and which coincides with $\ha\Pi$ outside the $U_k$.
One easily checks that $\Pi$ satisfies the hypotheses of the lemma.
\end{proof}

For $(\JJ,e)\in \JJ(\VV)\times H(\VV)\setm\big\{(\JJ(\rho(e),e)\,|\, e\in\,  ]e_0-\de_0,e_0+\de_0[\big\}$, we set 
\[
\ga_{\JJ,e}:=\Tt\lio\cap \Pi,
\]
where $\Tt\lio$ is the Liouville torus with $\JJ(\rho,e)=\JJ$.
For $z\in \VV$, we set $\ga(z):=\ga_{\JJ(z), H(z)}$.

The function 
\begin{equation}\label{actionI}
\begin{aligned}
A: & \:\:\VV  & \longrightarrow  & \:\:\R\\
      &\:\: z & \longmapsto & \:\:\int_{\ga(z)}\lam
\end{aligned}
\end{equation}
is well defined and $C^2$. Obviously, $A$ only depends on $e$ and $\rho$ and one immediately checks that $A$ is independent of $\II$.  By construction its vector field $X^A$ is $1$-periodic. 
We denote by $S$ the proper section given by proposition \ref{sectionsur Da} and by $\tau$ its associated transition function.

\vspace{0.3cm}

\noindent \textbf{Continuation of $S$ to  $\PP$.}
For $0<u<\bu$, and $1\leq k\leq p$, we define the surfaces $\Ga_k^-(u):=\{u_k=u\}\subset U_k\cap \DD_a$ and $\Ga_k^+(-u):=\{u_k=-u\}\subset U_k\cap \DD_a$. 
The continuation of $S$ in $\PP$ necessitates five steps.
\vspace{0.2cm}

\noindent $\bullet$ \textbf{Step 1:} \textit{For $1\leq k\leq p$, there exists  $u_k^*>0$ such that the Poincar\'e maps  associated with $\phi^A$ between the surfaces $\Ga_k^+(-u_k^*)$ and $ \Ga_k^-(u_k^*)$  are well defined and read}
\[
P_k(\vp_k,\us_k,\rho)=(\vp_k +\vth_k(\rho),\us_k, \rho)\quad 1\leq k\leq p,
\]
\textit{where $\vth_k:[0,a]\rit \R$ is $C^0$, $C^1$ on $]0,a]$ and satisfies $\lim_{\rho\rit 0}\vth_k(\rho)=0$}.

\begin{proof}
For $(e,\rho)\in H(\VV)\times F(\VV)$, $k\in \{1,\dots,p\}$, and $u\in\, ]0,\bu]$ we set
\[
\ga_{\rho,e}(k,u):=\ga_{\JJ(\rho,e),e}(k,u) :=(\ga_{\JJ(\rho,e),e}\cap\{[-u,u]\})\subset (\ga_{\JJ(\rho,e),e}\cap (U_k\cap\EE_e)),
\]
so that $\ga_{\rho,e}(k,u)$ is the part of $\ga_{\rho,e}$ limited by the sections $\Ga_k^\pm$ inside $U_k$.
Since $\rho\ma \JJ(\rho,e)$ is a diffeomorphism for fixed $e$, we will work with $\JJ$ in the rest of this proof and we  will write $\ga\je$ instead of 
$\ga_{\JJ(\rho,e),e}$.
In the coordinates $(\vp_k,I_k,u_k,v_k)$ introduced in (\ref{coordu_kv_k})  the $1$-form $\lam_k$ reads $\lam_k:=I_k d\vp_k +v_kdu_k$. Therefore,
\[
\int_{\ga_{\JJ,e}(k,u)}\lam_k=\int_{-u}^u\sqrt{s^2+2\JJ}ds.
\]
In particular, when $\JJ=0$:
\[
\int_{\ga_{0,e}(k,u)}\lam_k=\int_{-u}^u|s|ds=u^2.
\]
For $\ti{u},\us$ in $]0,\bu]$, we set 
\[
\Rr_{\ti u, \us}(\JJ,e):=\int_{\ga\je}\lam-\int_{\ga\je(1,\ti u)}\lam_1-\sum_{k=2}^p\int_{\ga\je(k,\us)}\lam_k.
\]
One immediately checks that $\Rr_{\ti u, \us}$ is $C^2$.
Then, writing $A$ as function of the variables $\JJ$ and $e$:
\[
A(\JJ,e):=\int_{\ga\je(1,\ti u)}\lam_1+\sum_{k=2}^p\int_{\ga\je(k,\us)}\lam_k +\Rr_{\ti u, \us}(\JJ,e).
\]
Set, for $e\in\, ]e_0-\de_0, e_0+\de_0[$,  $u_1(e):=\us+\frac{(e-e_0)}{2\us}\displaystyle\Dp{A}{e}(0,e_0)$. 
Then, one immediately checks that
\begin{equation}\label{R'=0}
\Dp{\Rr_{u_1(e), \us}}{e}(0,e_0)=\Dp{A}{e}(0,e_0)-2\us\frac{1}{2\us}\Dp{A}{e}(0,e_0)=0.
\end{equation}

Observe that $u_1(e_0)=u^*$. We set  $u_k^*=\us$, for $2\leq k \leq p$.
In the following, we will omit the lower indices $u_1(e), \us$ and write $\Rr(\JJ,e)$ instead of $\Rr_{u_1(e), \us}(\JJ,e)$.

In the coordinates, $(\vp_k, I_k,q_k, p_k)$ the vector field $X^A$, restricted to $U_k$, reads
\[
\hspace{-4.5cm}\dot\vp_k  =\Dp{\Rr}{e}(\JJ,e)\Dp{H}{I_k}(I_k,\JJ),\qquad\qquad \dot I_k=0
\]
\begin{multline*}
 \dot q_k  =\Dp{}{p_k}\int_{-u_1(e)}^{u_1(e)}\sqrt{s^2+2\JJ}ds +\sum_{\ell=2}^p\Dp{}{p_k}\int_{-\us}^{\us}\sqrt{s^2+2\JJ}ds
 \\
 +\Dp{\Rr}{\JJ}(\JJ,e)\Dp{\JJ}{p_k}(p_k,q_k)+\Dp{\Rr}{e}(\JJ,e)\Dp{H}{\JJ}(I_k,\JJ)\Dp{\JJ}{p_k}(p_k,q_k)
 \end{multline*}
 \begin{multline*}
\textcolor{white}{a}=\bigg(\int_{-u_1(e)}^{u_1(e)}\frac{ds}{\sqrt{s^2+2\JJ}}+ \sum_{\ell=2}^p\int_{-\us}^{\us}\frac{ds}{\sqrt{s^2+2\JJ}}\\+\Dp{\Rr}{\JJ}(\JJ,e)+\Dp{\Rr}{e}(\JJ,e)\Dp{H}{\JJ}(I_k,\JJ)\bigg)q_k,
  \end{multline*}
\begin{multline*}
 \dot p_k  =-\Dp{}{q_k}\int_{-u_1(e)}^{u_1(e)}\sqrt{s^2+2\JJ}ds +\sum_{\ell=2}^p\Dp{}{q_k}\int_{-\us}^{\us}\sqrt{s^2+2\JJ}ds
 \\
 +\Dp{\Rr}{\JJ}(\JJ,e)\Dp{\JJ}{q_k}(p_k,q_k)+\Dp{\Rr}{e}(\JJ,e)\Dp{H}{\JJ}(I_k,\JJ)\Dp{\JJ}{q_k}(p_k,q_k)
 \end{multline*}
 \begin{multline*}
\textcolor{white}{a}= -\bigg(\int_{-u_1(e)}^{u_1(e)}\frac{ds}{\sqrt{s^2+2\JJ}}+ \sum_{\ell=2}^p\int_{-\us}^{\us}\frac{ ds}{\sqrt{s^2+2\JJ}}\\+\Dp{\Rr}{\JJ}(\JJ,e)+\Dp{\Rr}{e}(\JJ,e)\Dp{H}{\JJ}(I_k,\JJ)\bigg)p_k.
  \end{multline*}
From now on, we will limit ourselves to the level $H=e_0$. Set
\[
\ka(\JJ):=\bigg(p\int_{-\us}^{\us}\frac{ds}{\sqrt{s^2+2\JJ}}+\Dp{\Rr}{\JJ}(\JJ,e_0)+\Dp{\Rr}{e}(\JJ,e_0)\Dp{H}{\JJ}(I_k,\JJ)\bigg)
\]
  and $R_k(\JJ):=\Dp{\Rr}{e}(\JJ,e_0)\Dp{H}{I_k}(I_k,\JJ)$. Then, in the coordinates $(\vp_k,q_k,p_k)$, the restriction of $X^A$ to $\DD_a$ reads:
  \[
  \dot \vp_k= R_k(\JJ),\quad \dot q_k=\ka(\JJ)q_k,\quad \dot p_k=-\ka(\JJ)p_k.
  \]
 Consider the renormalized vector field $X$ on $\DD_a$ defined by 
  \[
  X=\frac{1}{\ka(\JJ)}X^A.
  \] 
 Its flow $\phi$ has the same orbits as the flow $\phi_A$ associated with $X^A$, so the Poincar\'e maps between $\Ga_k^+(u_k^*)$ and 
 $\Ga_k^-(u_k^*)$ associated with $\phi$ and $\phi_A$ do coincide. 
 Observe that, if we introduce the transition time
  \[
 \quad T(\JJ):=\int_{-\us}^{\us}\frac{ds}{\sqrt{s^2+2\JJ}},\qquad 1\leq k\leq p,
  \]
then the map $P_k:z\ma \phi(T(\JJ),z)$  for  $z\in \Ga_k^+(u_k^*)$ is the Poincar\'e map associated with $X$ between $\Ga_k^+(u_k^*)$ and $\Ga_k^-(u_k^*)$.
  
Observe also that  $\JJ$ and  $A$ are in involution. Indeed, the orbits of $\phi_A$ are contained in the level sets $H=e, \JJ=J$. Therefore, $\JJ$ is also constant along the orbits of  $X$. Hence, one gets
\begin{multline}
P_k(\vp_k,\us,\rho)=\left(\vp_k+\int_0^{T(\JJ(\rho))}\frac{R_k(\JJ(\rho))}{\ka(\JJ(\rho))}dt,\us,\rho\right)\\
=\left(\vp_k+\frac{T(\JJ(\rho))R_k(\JJ(\rho))}{\ka(\JJ(\rho))},\us,\rho\right).
\end{multline}
Now
\[
\ka(\JJ):=\sum_{\ell=1}^p T(\JJ)+ B(\JJ),
\]
where $B: \JJ\ma\Dp{\Rr}{\JJ}(\JJ,e_0)+\Dp{\Rr}{e}(\JJ,e_0)\Dp{H}{\JJ}(I_k,\JJ)$ is a bounded function.
Since $\lim_{\JJ\rit 0}T(\JJ)=+\infty$ and $R_k(0)=0$, one gets $\lim_{\JJ\rit 0}\displaystyle\frac{T(\JJ)R_k(\JJ)}{\ka(\JJ)}=0$,
which concludes the proof with $\vth_k:\rho\ma \displaystyle\frac{T(\JJ(\rho,e_0))R_k(\JJ(\rho,e_0))}{\ka(\JJ(\rho,e_0))}$.
\end{proof}

In the following, we denote by $\Ga_k^\pm$ the surfaces $\Ga_k^\pm(u_k^*)$ and  by $P_k$ the Poincar\'e map associated with $\phi_A$ between 
$\Ga_k^+$ and $\Ga_k^-$. 
We denote by $D_k$ the subdomain of $\DD_a\cap U_k$ bounded by $\Ga_k^+$ and $\Ga_k^-$, that is, the domain $D_k:=\{|u|\leq \us_k\}$. Observe that $(\vp_k, u, \rho)$ is a system of coordinates in $D_\ell$.

For $1\leq k \leq p$, we call {\emph admissible arc}  on $\Ga_k^\pm$  a curve 
\[
\ze:=\{(\vp_k(\rho)),\mp\us,\rho))\,|\,\rho\in [0,a]\},
\]
which is $C^1$ on $]0,a]$ and $C^0$ in $[0,a]$.
\vspace{0.2cm}

\noindent \textbf{Step 2:} \textit{Let $\ze$ be an admissible arc on $\Ga_k^+$. Then $P_k(\ze)$ is an admissible arc on $\Ga_k^-$.
 Moreover, $\phi_A([0,1], \ze)\cap D_k$ has the following graph form}:
\[
\phi_A([0,1], \ze)\cap D_\ell:=\left\{\vp_k(u,\rho),u,\rho)\,|\, (u,\rho)\in[-u^*,u^*]\times[0,a]\right\}.
\]

\begin{proof}
Let $\ze$ be an admissible arc on $\Ga_k^+$. Then, $P_k(\ze):=(\vp_k(\rho)+\vth_k(\rho),\us, \rho)$. Now, by step 1, the function $[0,a]\rit \R:\rho\ma \vth_k(\rho)$ is $C^0$ on $[0,a]$ and $C^1$ on $]0,a]$, so $P_k(\ze)$ is an admissible arc.

To see the second part, we first observe that, in $D_k$, $\dot u>0$. Then for any $u_0\in [-\us,\us]$, the Poincar\'e map (associated with $\phi_A)$ between $\Ga_k^+$ and the surface $\Ga(u_0):=\{u=u_0\}$ is well defined. As before, this Poincar\'e map coincides with the Poincar\'e map associated with the flow $\phi$ and we denote by $T_{k,u_0}$ the associated time of the last one.
Then
\[
T_{k,u_0}:=\int_{-\us}^{u_0}\frac{ds}{\sqrt{s^2+2\JJ}}.
\]
which yields
\begin{multline}
\phi_A([0,1], \ze)\cap D_k:\\=\left\{\vp_k(\rho)+\frac{T_{k,u}(\JJ(\rho,e_0))R(\JJ(\rho,e_0))}{\ka(\JJ(\rho,e_0))},u,\rho)\,|\, (u,\rho)\in [-\us,\us]\times[0,a]\right\}.
\end{multline}
As before, one immediately check that 
\[
(u,\rho)\ma \vp_k(\rho))+\displaystyle\frac{T_{k,u}(\JJ(\rho,e_0))R(\JJ(\rho,e_0))}{\ka(\JJ(\rho,e_0))}
\]
is continuous on $[0,a]$ and $C^1$ on $]0,a]$.
\end{proof}
\vspace{0.2cm}

\noindent \textbf{Step 3:} \textit{The flow $\phi_A$ induces a Poincar\'e map $P_{k,k+1}$  between $\Ga_k^-$ and $\Ga_{k+1}^+$. Moreover, if $\zeta$ is an admissible arc on $\Ga_k^-$, $P_{k,k+1}(\ze)$ is an admissible arc on $\Ga_{k+1}^+$.}

\begin{proof}
First, we show that the flow $\phi_H$ defines a Poincar\'e map between $\Ga_k^+$ and $\Ga_k^-$ with associated transition-time $\sig_{k,k+1}$ that only depends on $\rho$.

We set $C_k^+=\Ga_k^+\cap W_k^+$ and $C_k^-=\Ga_k^-\cap W_k^+$. 
Note that the $\om$-limit set of $C_k^-$ with respect to $\phi_H$ is the hyperbolic orbit  $\Cc_{k+1}$. 
Now, since $C_k^-$ and $\Cc_{k+1}$ are not in the same connected component of $W_k^-\setm C_{k+1}^+$,  for all $z\in C_k$ there exists $\sig(z)>0$ such that $\phi_H(\sig(z),z)\in C_{k+1}^+$. 
Since $\dot{u}=\lam_k(\rho)\sqrt{u^2+\mu_k(\rho)}>0$ in $\DD_a\cap\UU_{k}$, the intersection time $\sig(z)$ is unique.
In the same way, for any $z\in C_{k+1}^+$, there exists a unique $\sig(z)<0$ such that $\Phi_H(\sig(z),z)\in C_{k}^-$. 
Therefore the Poincar\'e map between $C_k^-$ and $C_{k+1}^+$ is well defined. By compactness of $C_k^-$, there exist $0<t_1<t_2$ such that the associated  transition time $\sig$ takes its values in $[t_1,t_2]$.

One easily deduces from the previous study that, for $a>0$ small enough, the Poincar\'e map between $\Ga_k^-\cap\{\rho\in [0,a]\}$ and $\Ga_{k+1}^+\cap\{\rho\in [0,a]\}$ is well defined. We denote by $\sig_{k,k+1}$ its transition time. The function $\sig_{k,k+1}$ is smooth. 

Now, observe that the sections $\Ga_{k}^-$ and $\Ga_{k+1}^+$ are invariant under the flow $(\phi_{\II}^t)$ associated with the first integral $\II$.
 Indeed, the flow  $X^{\II}$ reads (in the variables $(\vp_k,p_k,q_k)$): $\dot\vp_k=1,\, \dot p_k=0=\dot q_k$, which yields
\[
\dot\vp_k=1,\quad \dot u=0,\quad \dot \rho=0,
\]
Moreover, since $\phi_H$ and $\phi_{\II}$ commute, for any $t>0$ and any $\vp_k\in \T$:
\begin{multline}
\phi_H(\sig_{k,k+1}(\vp_k,\us,\rho),(\vp_k+t,\us,\rho))\\=\phi_H(\sig_{k,k+1}(\vp_k,\us,\rho),\phi_{\II}(t,(\vp_k,\us,\rho)))\\
=\phi_{\II}(t,\phi_H(\sig_{k,k+1}(\vp_k,\us,\rho),(\vp_k,\us,\rho)))
\end{multline}
which belongs to $\Ga_{k+1}^+$.
Therefore $\sig(\vp_k+t,\us,\rho)=\sig(\vp_k,\us,\rho)$. We set $\sig_{k,k+1}(\rho):=\sig_{k,k+1}(\vp_k,\us,\rho)$, for any $\vp_k\in \T$.

We denote by $\overset{\circ}{D}_{k,k+1}$ the connected component of $\DD_a\setm(\Ga_k^-\cup\Ga_{k+1}^+)$ that does not contain $\Cc_k$ and we set $D_{k,k+1}:=\overset{\circ}{D}_{k,k+1}\cup\Ga_k^-\cup\Ga_{k+1}^+$.

Set $\Ga_k:=\{u=0\}\subset D_k$. The transition time $\sig_k$ (associated with $\phi_H$) between $\Ga_k\setm\{\rho=0\}$ and $\Ga_k^+\setm\{\rho=0\}$ is well defined and is independent of $\vp_k$. Indeed, 
\[
\sig_k(\rho):=\int_{-\us}^0\frac{ds}{\lam_k(\rho)\sqrt{s^2+\mu_k(\rho)}}.
\]
Since $\lim_{\rho\rit 0}\sig_k(\rho)=+\infty$, there exists $a>0$ small enough so that for any $1\leq k\leq p$, $\phi_H([0,1], \Ga_k^+)\subset \{u\in ]-\us, 0]\}$. Therefore
the map
\[
\begin{array}{lll}
\Ga_k^-\times [0,2] &\rit & D_{k,k+1}\cup\phi_H(]0,1], \Ga_k^+)\\
 ((\vp_k,\us,\rho),x) &\ma & \phi_H(x\sig_{k,k+1}(\rho),(\vp_k,\us,\rho))
\end{array}
\]
 is a diffeomorphism.
In the following, we consider $(\vp_k,x,\rho)$ as a system of coordinates in $D_{k,k+1}\cup\phi_H(]0,1], \Ga_k^+)$. 

As before, we will work with the vector field $X$ and its flow $\phi$ instead of $X^A$ and $\phi_A$, and we will show that $\phi$ induces a Poincar\'e map between $\Ga_k^-$ and $\Ga_{k+1}^+$.
Since $\rho$ is invariant under $\phi$, the vector field $X$ reads:
\[
X(\vp_k,x,\rho):=(X_\vp(\vp_k,x,\rho), X_x(\vp_k,x,\rho), 0),
\]
where $X_\vp$ and $X_x$ are $C^1$ functions.

Now, let $u\in [\us, \bu]$. Then, 
\[
x(\vp_k,u,\rho)=\frac{1}{\sig_{k,k+1}(\rho)}\int_{\us}^u\frac{ds}{\sqrt{s^2+2\JJ(\rho)}},
\]
that is, $x$ is a strictly increasing function of $u$ in $D_{k,k+1}\cap \UU_k$. Hence, since $\dot u_k=p_k+q_k>0$ on $\Ga_k^-$,
\begin{equation}\label{X_x(0)nonnul}
X_x(\vp_k,0,\rho)>0,\quad \forall\, (\vp_k,\rho)\in \T\times [0,a].
\end{equation}
Moreover, since $x$ is independent of $\vp_k$ in the domain $\{u\in [\us, \bu]\}$, 
for any $\vp_k \in \T$, $X_x(\vp_k,0, \rho)=X_x(0,0,\rho)$.

With $x_0\in [0,1]$, we associate the diffeomorphism
\[
\begin{array}{llll}
\eta_{x_0} : & D_{k,k+1} & \rit  & D_{k,k+1}\cup\phi_H([0,1], \Ga_k^+)\\
 & (\vp_k,x,\rho) & \ma & (\vp_k,x+x_0,\rho).
\end{array}
\]
Observe that 
for any $\vp_k\in \T$, $D_{(\vp_k,0,\rho)}\eta_{x_0}(X(\vp_k,0,\rho)=X(\vp_k,x_0,\rho))$. So,  for any $\vp_k\in \T$, $X(\vp_k,x_0,\rho)=X(0,x_0,\rho)$. This property holds for any $x\in [0,1]$. As a consequence, using (\ref{X_x(0)nonnul}), one sees that  there exists $\be>0$ such that, for all $z\in D_{k,k+1}$, $X_x(z)>\be$.
Therefore, if we set
\[
t_{k,k+1}(\rho)=\int_0^1\frac{ds}{X_x(0,s,\rho)},
\]
the map 
\[
\begin{array}{llll}
P_{k,k+1}: & \Ga_k^- & \rit  & \Ga_{k+1}^+\\
 & (\vp_k,0,\rho) & \ma & \phi_A(t_{k,k+1}(\rho),\vp_k,0,\rho)).
\end{array}
\]
is a Poincar\'e map between $\Ga_k^-$ and $\Ga_{k+1}^+$.
We denote by $\vp_A:\T\times[0,a]\rit\T$ the smooth function defined by  $\phi_A(t_{k,k+1}(\rho)(\vp_k,0,\rho)=(\vp_A(\vp_k,\rho),1,\rho)$ and by $\vp_H:\T\times[0,a]\rit\T$ the smooth function defined by  $\phi_H(\sig_{k,k+1}(\rho)(\vp_k,\us,\rho)=(\vp_H(\vp_k,\rho),-\us,\rho)$.
Hence, 
\[
P_{k,k+1}(\vp_k,\us,\rho)=(\vp_H(\vp_A(\vp_k,\rho),\rho),-\us,\rho).
\]
As a consequence, if $\ze:=\{(\vp_k(\rho),\us, \rho)\,|\, \rho\in [0,a]\}$ is an admissible arc on $\Ga_k^-$, its image $P_{k,k+1}(\vp_k(\rho),\us,\rho)=(\vp_H(\vp_A(\vp_k(\rho),\rho),\rho),-\us,\rho)$ is  an admissible on $\Ga_{k+1}^+$.
\end{proof}

Consider now the surface $S:=\phi_A([0,1], \ze_1)$ as defined in proposition \ref{sectionsur Da}. Observe that, according to the previous graph form
(Step 2),  the vector field $X^H$ is transverse to $S$ inside the domains $D_k$. By construction of $S$, this immediately implies that $X^H$ is everywhere transverse to $S$. Indeed, the action-angle form proves that it is enough that $X^H$ be transverse to $S$ at only one point in each
Liouville torus of the regular foliation. So $S$ is a proper section for $X^H$.

\vspace{0.2cm}

\noindent \textbf{Step 4:} \textit{The surface $S$ can be continuated to $\PP$ as a continuous surface}.

\begin{proof}
Using alternatively Steps 1 and 3, we see that $\ze_k^\pm:=\phi_A([0,1],\ze_1)\cap \Ga_k^\pm$ is an admissible arc for any $1\leq k\leq p$. 

Using step 2, one sees that $\phi_A([0,1],\ze_1)\cap D_k$ is a continuous submanifold with the graph form
$
\phi_A([0,1], \ze_1)\cap D_k:=\left\{\vp_k(u,\rho),u,\rho)\,|\, (u,\rho)\right\}.
$

It remains to check that $\phi_A([0,1], \ze_1)\cap D_{k,k+1}$ is a continuous submanifold.
For $\rho\in [0,a]$, we write $\ze_k^-(\rho):=(\vp_k(\rho), \us, \rho)$.
Observe that
\[
\phi_A([0,1], \ze_1)\cap D_{k,k+1}=\bigcup_{\rho\in [0,a]}\bigcup_{x\in [0,1]}\phi_X(xt_{k,k+1}(\rho),\ze_k^-(\rho))
\]
The map $[0,a]\times[0,1]\rit S : \phi_X(xt_{k,k+1}(\rho),\ze_k^-(\rho))$ is a homemorphism onto $\phi_A([0,1], \ze_1)\cap D_{k,k+1}$, which is $C^1$ on $[0,1]\times ]0,a]$. This proves that $S$ admits a $C^0$ continuation on $\rho=0$.
\end{proof}

We still denote by $S$ the continuation of $S$ in $\ov{\DD}_a$.
\vspace{0.2cm}

\noindent \textbf{Step 5:} \textit{The surface $S$ is a proper section for $\phi_H$ associated with a $C^1$ function $\tau:[0,a]\rit \R_+^*$}.\index{$\tau(\rho)$}

\begin{proof}
We will first show that the transition function $\tau:\,]0,a]\rit \R$ defined in proposition \ref{sectionsur Da} has a well defined $C^1$ continuation on $[0,a]$. Then we will prove that this function is a transition function for $S$ with respect to $\phi_H$.

We consider the lift $\R\times [-\us, \us]\times [0,a]$ of $\T\times [-\us, \us]\times [0,a]\subset \UU_k$. We still denote by $\phi_H$,  $\phi_A$ and $\phi$ the lifted Hamiltonian flows  on $\R\times [-\us, \us]\times [0,a]$. Let us denote by $T$ the common period of the orbits $\Cc_k$.
We set $X^H(z):=(X^H_\vp(z), X^H_u(z), 0)$.

By continuity of $X^H$, we can assume that $\us$ and $a$ are small enough so that for any $z\in\R\times [-\us, \us]\times [0,a]$,\, $X_\vp^H(z)\in [\frac{2}{T}, \frac{1}{2T}]$.
Hence, there exists $\al>0$ such that
\[
\phi_H(\dfrac{T}{2},\ze_1)\subset \{\vp_1\leq 1-\al\},\quad \phi_H(2T,\ze_1)\subset \{\vp_1\geq 1+\al\}.
\]
On the other hand,  we can also assume that $a$ is small enough so that there exists $u_0\in \,]0, \us]$ such that
\[
\phi_H([0,2T], \ze_1)\subset \{u\in [-\us,-u_0]\}.
\]
Let $\ti\ze_1:=\ze_1+(1,0,0)\subset\R\times [-\us, \us]\times [0,a]$ and set  $\ti\ze_1(\rho):=(1, -\us, \rho)$.
Finally set $t(\rho):=\displaystyle\int_{-\us}^{-u_0}\frac{ds}{ \sqrt{u^2+2\JJ(\rho)}}$. 
By construction,  since in the coordinates $(\vp_1, u,\rho)$ the vector field $X$ reads
\[
\dot \vp_1=\frac{R(\JJ(\rho))}{\ka(\JJ(\rho))},\quad \dot u= \sqrt{u^2+2\JJ(\rho)}, \quad \dot \rho =0,
\]
for any $\rho\in [0,a]$, $\phi_A(t(\rho), (1, -\us, \rho))\in \{u=u_0\}$.
Let $\wti S_1$ be the smooth surface 
\[
\wti S_1:=\bigcup_{\rho\in [0,a]}\phi_X([0,t(\rho)], \ti\ze_1(\rho)).
\]
By step 2, $\wti S_1$ has the following graph form:
\[
\wti S_1:=\{(\vp_1(u,\rho), u, \rho)\,|\, (u,\rho)\in [-\us, -u_0]\times [0,a]\}
\]
Since the function $\rho\rit t(\rho)$ is decreasing, therefore 
\[
\wti S_1:=\phi_X([0,t(0)], \ti\ze_1)\cap \{u\in [-\us,-u_0]\}.
\]
Moreover, since $\lim_{\rho\rit 0}\frac{R(\JJ(\rho))}{\ka(\JJ(\rho))}=0$, there exists $a>0$ small enough so that for any $\rho\in [0,a]$, $t(0)\displaystyle\frac{R(\JJ(\rho))}{\ka(\JJ(\rho))}<\al$.
As a consequence, 
\[
\wti S_1\subset B:= [1-\al,1+\al]\times [-\us, -u_0]\times [0,a],
\]
and  $B\setm\wti S_1$ has two connected components.
Hence for any $\rho \in [0,a]$, $\phi_H(\R,(0, \-us, \rho))$ cut $\wti S_1$ once and only once. 
Let $\tau:[0,a]\rit \R^+$ be such that
$
\phi_H(\tau(\rho), (0,-us, \rho)\in \wti S_1.
$
Then $\tau$ is the restriction to $\{-\us\}\times [0,a]$ of the transition time associated to the Poincar\'e map (with respect to $\phi_H$) between the smooth surfaces $\{\vp=0\}$ and $\wti S_1$. This  map  is  smooth  since the both surfaces are smooth and transverse to $\phi_H$. Hence, $\tau$ is  smooth on $[0,a]$.
Finally, if $\pi :\R\rit \T$ is the canonical projection, $\pi (\wti S_1)\subset S$. 
Obviously, the function $\tau$ defined above coincide in $S\cap\{\rho>0\}$ with the transition function $\tau$ defined in proposition \ref{sectionsur Da}. 
It remains to check, that for any $z\in S\cap\PP$, the transition time $\tau(z)$ of $z$ is $\tau(0)$.
Fix $z\in S\cap \PP$. Let $\ga :[0,a]\rit S$ be a continuous map such that $\ga(0)=z$ and for any $r\in \,]0,a]$, $\ga(r)\subset S\cap\{\rho=r\}$.
Then for any $\rho\in\, ]0,a]$, $\phi_H(\tau(\rho), \ga(\rho))\in S$. Since $S$ is closed, by continuity of $\phi_H, \tau$ and $\ga$, $\ga(z)\in S$.
\end{proof}

\subsubsection{Construction of the conjugacy between $\wp$ and $\psi^\al$: proof of Lemma II}

This section is devoted to the proof of lemma II. Let $S$ be a proper section with time $\tau$ as in lemma I. In each $\UU_k\cap \ov{\DD_a}$, $S$ has the following graph form
\[
S:=\{(\vp_k(u,\rho)\,w,\rho)\,|\, \vp\in \T, (u,\rho)\in[-\ov{u},\ov{u}]\times [0,a]\}.
\]
The proof consists in the construction of a suitable planar $p$-model $\psi$. Our strategy is the following.
\vspace{0.15cm}

\noindent $\bullet$ We first construct a \emph{fundamental domain}  for the map $\wp$ in each subdomain $\UU_k\cap S$ of $S$. A fundamental domain means a domain $\De_k$ bounded by a ``vertical'' curve $D_k$ with equation $u=u_0$ in $S$ and its  image $\wp\inv(D_k)$ (and the natural horizontal boundaries).
\vspace{0.15cm}

\noindent $\bullet$ We show that there exists an integer $m_k$ such that the domains $\De_k,\, \wp(\De_k),\,$ $\dots,\, \wp^{m_k}(\De_k)$ cover the connected component $R_k$ of $S$ between $S\cap \UU_k$ and $S\cap\UU_{k+1}$ and such that $\wp^{m_k}(\De_k)\subset S\cap\UU_{k+1}$.
\vspace{0.25cm}

This being done, the construction of the planar $p$-model necessitates three steps.
\vspace{0.2cm} 

\noindent $\bullet$ For $1\leq k\leq p$, we construct a vector field $X_k$ with associated flow $\psi_{(k)}$ in a suitable neighborhood $\OO_k$ of the  point $(\kp,0)\in \T\times [0,1]$, such that there exists a homeomorphism $\chi_k$ between $\De_k$ and a fundamental domain  of $X_k$.
\vspace{0.15cm}

\noindent $\bullet$ For $1\leq k\leq p$, we construct a vector field of $\ha{X}_k$ with associated flow $\ha{\psi}_{(k)}$ in  a suitable  subdomain $\RR_k$ of $\T\times [0,1]$ such that we can glue together the flows $\psi_{(k)}$ and $\ha{\psi}_{(k)}$ (for $1\leq k\leq p$) to get a flow $\psi$ on $\T\times[0,1]$ with a time map $\psi^{\al}$ conjugated to $\wp$.
\vspace{0.15cm}

\noindent $\bullet$ We check that the flow $\psi$ is a planar $p$-model.

\paraga{\it Construction of the fundamental domains $\De_k$.}
The construction of the fundamental domains $\De_k$\index{$\De_k$} is based on the construction in each domain $\UU_k$ of a pair of sections  that are transverse both  to the flow and to $S$.

\begin{lem} There exists $\us\in [0,\bu]$ such that, for any $1\leq k\leq p$, the sections $\Sig_k^+ :=\{u=-\us\}$ and  $\Sig_k^-:=\{u=\us\}$ satisfy the following conditions:

\noindent $\bullet$ \textbf{(C1):} The Poincar\'e return map  between $\Sig_k^+\setm\{\rho=0\}$ and $\Sig_k^-\setm\{\rho=0\}$ is well defined and its associated time $\tau_k$ does not depend on $\vp$ and is a decreasing function of $\rho$.
\vspace{0.1cm}

\noindent $\bullet$ \textbf{(C2):} The Poincar\'e return map  between   $\Sig_k^-$ and $\Sig_{k+1}^+$ is well defined and its associated time $\sig_{k,k+1}$  does not depend on $\vp$ and is a decreasing function of $\rho$.\index{$\sig_{k,k+1}$}
\end{lem}

\begin{proof}
Let us prove (C1). Fix $u_0\leq \bu$ and consider the surfaces $\Sig_k^\pm$ defined as above for $1\leq k \leq p$. By $(\star)$, one immediately checks that they are transverse to the flow and that the transition time between $\Sig_k^+\setm\{\rho=0\}$ and $\Sig_k^-\setm\{\rho=0\}$ is given by 
\[
\tau_k(\rho)=\frac{2}{\lam(\rho)}\Argsh\frac{u_0}{\sqrt{\mu(\rho)}}.
\]
Moreover, by direct computation, one sees that $\lim_{\rho\rit 0}\tau_k'(\rho)=-\infty$.
As a consequence, if $a$  is small enough, $\tau'_k(\rho)<0$ for any $\rho\in [0,a]$. In the following, we assume that $a$ is small enough so that, for any $1\leq k\leq p$, $\tau_k$ is decreasing on $[0,a]$ and (C1) is realized.

In Step 3 of the continuation of $S$, we proved the existence of a Poincar\'e map between $\Sig_k^-$ and $\Sig_{k+1}^+$ with associated transition time $\sig_{k,k+1}$ that only depends on $\rho$.
It remains to check the decreasing condition on the time $\sig_{k,k+1}$. 
Note that the function $\rho\ma\sig_{k,k+1}'(\rho)$ is uniformely bounded on $[0,a]$.
Fix $0<\us<u_0$ and consider, for $1\leq k \leq p$, the sections $\Sig_k^+(\us) :=\{u=-\us\}$ and $\Sig_k^-(\us) :=\{u=\us\}$. Let $\ti{\sig}_{k,k+1}$ be the transition time between $\Sig_k^-(\us)$ and $\Sig_{k+1}^+(\us)$. Then
\[
\ti{\sig}_{k,k+1}(\rho):=\frac{2}{\lam(\rho)}\left(\Argsh\frac{u_0}{\sqrt{\mu(\rho)}}-\Argsh\frac{\us}{\sqrt{\mu(\rho)}}\right)+\sig_{k,k+1}(\rho).
\]
Let $g:\rho\ma \frac{2}{\lam(\rho)}\left(\Argsh\frac{u_0}{\sqrt{\mu(\rho)}}-\Argsh\frac{\us}{\sqrt{\mu(\rho)}}\right)$.
By elementary computation one sees that:
\[
g'(\rho)\sim_{\rho\rit 0}\frac{\lam(\rho)\mu'(\rho)}{4}\left(\frac{1}{u_0^2}-\frac{1}{(\us)^2}\right).
\]
Therefore, for $\us$ small enough $\rho\rit \ti{\sig}_{k,k+1}(\rho)$ is decreasing and (C2) is also realized. Obviously, one can choose $\us$ small enough  so that (C1) and (C2) are realized for any $1\leq k\leq p$. 
\end{proof}

\begin{rem}\label{sig_k,k+1borneInf}
We can assume that $\us$ is small enough so  that, for any $k\in \{1,\cdots, p\}$ and any $\rho\in [0,a]$, $\sig_{k,k+1}(\rho)>2$.
\end{rem}

For $1\leq k\leq p$, we set $\de_k:=S\cap \Sig_k^-:=\{(\vp_k(\us,\rho)\,|\, \rho\in [0,a]\}$. Since for all $k$, $\lim_{\rho\rit 0}\tau_k(\rho)=+\infty$, there exists $a>0$ small enough such that for any $\rho\in [0,a]$ and any $k$, $\tau(\rho)\leq \tau_k(\rho)$.
Therefore, $\wp\inv(\de_k)\subset \UU_k\cap  S$.

The fundamental domain $\De_k$ is defined as the subdomain of $S\cap \UU_k$ bounded by $\wp\inv(\de_k)$ and $\de_k$. Fix $(\vp_k(\us,\rho),\us, \rho)\in \de_k$. Then 
\[
\wp\inv(\vp_k(\us,\rho),\us, \rho)=(\vp_k(u_k(\rho),\rho),u_k(\rho),\rho)
\]
where $u_k(\rho)$ is defined by 
\[
\tau(\rho):=\displaystyle\int_{u_k(\rho)}^{\us} \frac{du}{\lam_k(\rho)\sqrt{u^2+\mu_k(\rho)}}.
\]
Therefore,
\[
\De_k =\left\{(\vp_k(u,\rho),u,\rho)\,|\, u\in [u_k(\rho),\us], \rho\in\,[0,a]\right\}.
\]

We denote by $R_k$\index{$R_k$} the connected component of $S\setm(\UU_k\cup \UU_{k+1})$ that contains $\wp(\De_k)$, that is, the connected component of $S\setm(\UU_k\cap \UU_{k+1})$ with nonempty intersection with $W_k^-$.

\begin{lem}
For $1\leq k\leq p$, there exists $m_k\in \N^*$ such that
\begin{itemize}
\item  $\wp^{m_k}(\De_k)\subset (S\cap \UU_{k+1})$,
\item$R_k\subset\bigcup\limits_{j=1}^{m_k}\wp^j(\De_k)$.
\end{itemize}
\end{lem}

\begin{proof}
For $a>0$ we set $T_a:=\{\tau(\rho)\,|\,\rho \in [0,a]\}$. There exists $0<t_0\leq t_1$ such that $T_a:=[t_0(a),t_1(a)]$. Notice that, if $a'\leq a$, then $t_0(a)\leq t_0(a')\leq t_1(a')\leq t_1(a)$.
By compactness of $S_k$ and continuity of $\sig_{k,k+1}$, one can define $\sig_k:=\max_{z\in \Sig_k^+}\sig_{k,k+1}(z)>0$. 
There exists $m_k\in \N^*$ such that $(m_k-1)t_0(a)\geq \sig_k$. Since $\lim_{\rho\rit 0}\tau_{k+1}(\rho)=+\infty$, one can assume that $a$ is small enough so that $m_kt_1(a)<\tau_{k+1}(\rho)-\sig_{k,k+1}(\rho)$ for any $\rho\in [0,a]$. That is, 
\[
\Phi_H([(m_k-1)t_0(a),m_kt_1(a)], \Sig_k^-)\subset \UU_k.
\]
In particular,  for all $z\in \de_k$, $\wp^{m_k}(z)\subset \Phi(m_k[t_0(a),t_1(a)],\Sig_k^-)\subset \UU_k$, that is, $\wp^{m_k}(\de_k)\subset \UU_k\cap S$.
In the same way, for all $z\in \wp\inv(\de_k)$, $\wp^{m_k}(z)\subset \Phi((m_k-1)[t_0(a),t_1(a)],\Sig_k^-)\subset \UU_k$, that is, $\wp^{m_k-1}(D_k)\subset \UU_k\cap S$.

We set $\stackrel{\circ}{\De_k}:=\De_k\setm(\de_k\cup\wp\inv(\de_k))$. Since $\wp^{m_k}$ is a diffeomorphism, $\wp^{m_k}(\stackrel{\circ}{\De_k})$ is one of  the two connected components of $S\setm (\wp^(\de_k)\cup\wp^{m_k-1}(\de_k))$. Since the second one has nonempty intersection with all the hyperbolic orbits $\Ga_j$,  $\wp^{m_k}(\stackrel{\circ}{\De_k})$ must be the first one which is contained in $\UU_{k+1}$ and the first point is proved.  

To prove the second point, we first remark that $R_k$ is contained in the connected component of $S\setm(\wp(D_k)\cup \wp^{m_k-1}(D)$.
With the same argument as in the beginning of the proof, we can assume that $a$ is small enough so that $\phi_H([-m_k, 0], \De_k)\subset \UU_k$. 
So for $2\leq j\leq m_k$,
\[
\wp^{-j}(\de_k)=\{(\vp_k(w_j(\rho), \rho),w_j(\rho),\rho)\}
\]
where $u_j(\rho)$ is defined by  $\displaystyle\int_{u_j(\rho)}^{\ws} \frac{du}{u^2+\mu(\rho)}=j\tau(\rho)$. Therefore,
$\bigcup\limits_{j=1}^{m_k}\wp^{-j}(\De_k)$ is a connected $2$-dimensional submanifold of $S$ with boundaries $\wp^{-m_k}(\de_k)$ and $\wp\inv(\de_k)$. 
As a consequence, $\bigcup\limits_{j=1}^{m_k}\wp^j(\De_k)=\wp^{m_k}\left(\bigcup\limits_{j=1}^{m_k}\wp^{-j}(\De_k)\right)$ 
is a connected $2$-dimensional submanifold of $S$ with boundaries $\wp(\de_k)$ and $\wp^{m_k-1}(\de_k)$. Obviously, it is the one which contains $R_k$.
\end{proof}

Consider the compact annulus $\hjA:=\T\times[0,1]$ with coordinates $(\th,r)$. For $1\leq k\leq p$, we set $z_k:=(\kp,0)$,  $\OO_k:=[\kp-\us,\kp+\us]\times[0,1]$ and $\RR_k:=[\kp+\us,\frac{k+1}{p}-\us]\times[0,1]$.\index{$\RR_k$}
Finally we set 
\[\al:[0,1]\rit \R_+^* : r\ma \tau(ar).
\]

\paraga{\it Construction of $X_k$ and $\chi_k$.}
Let $X_k$\index{$X_k$} be the vector field defined on $\OO_k$ by 
\[
X_k(\th,r):=\lam_k(ar)\sqrt{\left(\th-\kp\right)^2+\mu_k(ar)}\Dp{}{\th}.
\]
We denote by $(\psi_{(k)}^t)$ its local flow in $\OO_k$ and by $\psi_{(k)}^\al$ its associated time-$\al$ map, that is,
$
\psi_{(k)}^\al(\th,r)=\psi_{(k)}(\al(r),(\th,r)).
$

For $1\leq k\leq p$, we set
\[
\begin{array}{llll}
\chi_k: &\UU_k\cap \Sig & \rit & \OO_k \\
& (\vp_k(u,\rho),u,\rho) & \ma & (u+\kp, \frac{\rho}{a}).
\end{array}
\]

The proof of the following lemma is immediate.

\begin{lem}\label{chi_k}
Let $k\in \{1,\cdots,p\}$. 
\begin{enumerate}
\item $\psi_{(k)}^\al(\chi_k(\wp\inv(\de_k)))=\chi_k(\de_k)$,
\item $\psi_{(k+1)}^\al(\chi_{k+1}(\wp^{m_k-1}(\de_k)))=\chi_{k+1}(\wp^{m_k}(\de_k))$,
\item for all $z\in \UU_k\setm\De_k$, $\psi_{(k)}^\al(z)=\chi_k(P(z))$.
\end{enumerate}
\end{lem}



\paragraph{Construction of $\ha X_k$ and of the flow $\psi$.}\index{$\ha X_k$} For any $k\in \{1,\cdots,p\}$, we consider a function $\xi_k:\RR_k\rit \R$ such that
\begin{equation}\tag{**}
\int_{\kp+\us}^{\kkp-\us}\frac{d\th}{\xi_k(\th,r)}=\sig_{k,k+1}(ar)
\end{equation}
Let $\ha{X}_k$ be the vector field on $\RR_k$ defined by 
\[
\ha{X}_k(\th,r)=\xi_k(\th,r)\Dp{}{\th},
\]
and denote by $(\ha{\psi}_{(k)})$ its local flow. 
By construction, for any $r\in [0,1]$
\[
\ha{\psi}_{(k)}\left(\sig_{k,k+1}(ar), \left(\kp+\us, r\right)\right)= \left(\kkp-\us, r\right),
\]
We define a flow $\psi$ on $\T\times[0,1]$ by gluing together the flows $\psi_{(k)}$ and $\ha{\psi}_{(k)}$. We begin by constructing, for $1\leq k\leq p$, a local flow $\psi_k$ on $\OO_k\bigcup\RR_k$ in the following way.
For $(\th,r)\in \OO_k\setm\{z_k\}$ there exists a unique $t(\th,r)>0$ such that $\psi_{(k)}(t(\th,r))\in \chi_k(\de_k)$. 
We set 
\begin{itemize}
\item $\psi_k^t(\th,r)=\psi_{(k)}(t,(\th,r))$ if $t\leq t(\th,r)$,
\item $\psi_k^t(\th,r)=\ha{\psi}_{(k)}(t-t(\th,r),(\kp+ \us,r))$ if $ t(\th,r)\leq t \leq t(\th,r)+ \sig_{k,k+1}(ar)$.
\end{itemize}

\begin{lem}
$\psi_k$ is a continuous flow on $\OO_k\bigcup\RR_k$.
\end{lem}

\begin{proof}
The continuity of $\psi_k$ is obvious by construction.
One just has to check that $\psi_k$ is a flow, that is, $\psi_k(s+t,(\th,r))=\psi_k^s\circ \psi_k^t(\th,r)=\psi_k^t\circ \psi_k^s(\th,r)$. The only possible difficulty occurs when 
$t(\th,r)<t+s<t(\th,r)+ \sig_{k,k+1}(ar)$.

Assume that $s\leq t$ and that $t(\th,r)<t+s<t(\th,r)+ \sig_{k,k+1}(ar)$. We set $\ths:=\kp+\us$. 
We first remark that $\psi_k(s+t,(\th,r))=\ha{\psi}_k(t+s-t(\th,r),(\ths,r))$.
They are three possibilities.
\vspace{0.15cm}

\noindent $\bullet$ $t<s<t(\th,r)$. Then $t(\psi_{(k)}(s,(\th,r)))= t(\th,r)-s$ and  $t>t(\psi_{(k)}(s,(\th,r)))$. Hence
\begin{align*}
\psi_k^t\circ\psi_k^s(\th,r)=\psi_k(t, \psi_k^s(\th,r)) & =\ha{\psi}_{(k)}(t-t(\psi_{(k)}(s,(\th,r))),(\ths,r)) \\
 & =\ha{\psi}_{(k)}(t+s-t(\th,r),(\ths,r)) \\
 & = \psi_k(s+t,(\th,r)).
\end{align*}
In the same way, $\psi_k^s\circ\psi_k^t(\th,r)=\psi_k(s+t,(\th,r))$.
\vspace{0.15cm}

\noindent $\bullet$ $t\leq t(\th,r)\leq s$. Then $\psi_k(s,(\th,r))=\ha{\psi}_{(k)}(s-t(\th,r), (\th,r))$, which yields
\begin{multline}
\psi_k^t\circ\psi_k^s(\th,r)=\ha{\psi}_{(k)}(t, \ha{\psi}_{(k)}(s-t(\th,r), (\ths,r)))\\
=\ha{\psi}_{(k)}(t+s-t(\th,r),(\ths,r))=\psi_k(s+t,(\th,r)).
\end{multline}
On the other hand, 
\begin{multline}
\psi_k^s\circ\psi_k^t(\th,r)=\ha{\psi}_{(k)}(s-t(\psi(s,(\th,r))), (\ths,r))\\
=\ha{\psi}_{(k)}(t+s-t(\th,r),(\ths,r))=\psi_k(s+t,(\th,r)).
\end{multline}
\vspace{0.15cm}

\noindent $\bullet$ $ t(\th,r)\leq t\leq s$. Then, as before $\psi_k^t\circ\psi_k^s(\th,r)=\psi_k(s+t,(\th,r))$. Conversely, 
\[
\psi_k^s\circ\psi_k^t(\th,r)=\ha{\psi}_{(k)}(s, \ha{\psi}_{(k)}(t-t(\th,r), (\ths,r)))=\psi_k(s+t,(\th,r)). 
\]
This concludes the proof.
\end{proof}

\paraga {\it Construction of $\psi$ and conjugacy between $\psi^\al$ and $\phi_H$.} Now, we construct a global flow  $\psi$ on $\jA$ by gluing together the previous flows defined on $\OO_k\bigcup\RR_k$ with the usual convention $p+1=1$. One checks as in the previous lemma that this defines a flow on $\T\times[0,1]$.
We denote by $\psi^\al$ its time-$\al$ map.

\begin{lem}\label{wp^m_k_et_chi_k}
For all $z\in \De_k$
\[
(\psi^\al)^{m_k}(\chi_k(z))=\chi_{k+1}(\wp^{m_k}(z)).
\]
\end{lem}

\begin{proof}
Fix $z:=(\vp_k(u,\rho),u,\rho)\in \De_k$. Let $u'\in [u_{k+1}(\rho), \us]$ be such that $\wp^{m_k}(z)=(\vp_{k+1}(u',\rho),u',\rho)$.
We set 
\[
t_1:=\displaystyle\int_{u}^{\us}\frac{ds}{\lam_k(\rho)\sqrt{s^2+\mu_k(\rho)}},\quad t_2:=\displaystyle\int_{-\us}^{u'}\frac{ds}{\lam_k(\rho)\sqrt{s^2+\mu_k(\rho)}}
\]
Then 
\begin{equation*}
\begin{aligned}
\chi_{k+1}(\wp^{m_k}(z)) &=\left(\kkp+u',\frac{\rho}{a}\right) \\
 &= \psi\left(\sig_{k,k+1}(\rho)+t_1(z)+t_2(z),\kp +u\right)\\
 &=\psi(\tau(\rho), \chi_k(z)).
\end{aligned}
\end{equation*}
This concludes the proof.
\end{proof}

We define a map $\chi: S\rit \T\times [0,1]$ by setting
\vspace{0.2cm}

\noindent $\bullet$ for $1\leq k\leq p$ and $z\in \OO_k$,\: $\chi(z)=\chi_k(z)$,
\vspace{0.2cm}

\noindent $\bullet$ for $1\leq k\leq p$ and $ z\in \wp^j(\De_k)$ with $1\leq j\leq m_k$,\: $\chi(z)=(\psi^\al)^j(\chi_k(\wp^{-j}(z)))$.
\vspace{0.2cm}

\begin{lem}
The map $\chi$ is well defined and is a homeomorphism. 
\end{lem}

\begin{proof}
To see that $\chi$ is well defined, one has to check that, for $1\leq k\leq p$,  both definitions of $\chi$ coincide when $z\in \wp^{m_k}(z)$.
Fix $z\in \wp^{m_k}(z)$ and let $z_0=\wp^{-m_k}(z)\in \De_k$. Then
\[
(\psi^{\al})^{m_k}(\chi_k(\wp^{-m_k}(\wp^{m_k}(z_0)))=(\psi^{\al})^{m_k}(\chi_k(z_0))=\chi_{k+1}(\wp^{m_k}(z_0))=\chi_{k+1}(z),
\]
the third equality coming from lemma \ref{wp^m_k_et_chi_k}.

\noindent Obviously, $\chi$ is continuous on $S\setm\left(\bigcup_{1=k}^p\bigcup_{1=j}^{m_k}\wp^j(\de_k)\right)$. Fix $k$ and $j\in \{1,\cdots, m_k\}$.
One just has to check that for $z_0\in \wp^j(\de_k)$, 
\[
\lim_{\substack{z\rit z_0\\z\in \wp^j(\De_k)}}\chi(z)=\lim_{\substack{z\rit z_0\\z\in \wp^{j+1}(\De_k)}}\chi(z).
\]
Let $z_1=\wp^{-j}(z_0)$. Then
\[
\lim_{\substack{z\rit z_0\\z\in \wp^j(\De_k)}}\chi(z)=(\psi^\al)^j(\chi_k(\wp^{-j}(z_0)))=(\psi^\al)^j(\chi_k(z_1)).
\]
On the other hand, using lemma \ref{chi_k}, one checks that:
\[
\lim_{\overset{z\rit z_0}{z\in \wp^{j+1}(\De_k)}}\chi(z)=(\psi^\al)^{j+1}(\chi_k(\wp^{-j-1}(z_0)))=(\psi^\al)^j(\chi_k(z_1)),
\]
By construction, $\chi$ is a homeomorphism.
\end{proof}

\begin{lem}
For all $z\in S$, $\chi \circ \wp(z)=\psi^\al \circ \chi(z)$.
\end{lem}

\begin{proof}
By lemma \ref{chi_k}, it remains to check the conjugacy when $z$ belongs to $\bigcup_{1=k}^p\bigcup_{j=0}^{m_k}\wp^j(\De_k)$. Fix $k\in\{1,\ldots,p\}$, $j\in \{0,\cdots m_k\}$  and $z\in \wp^j(\De_k)$.
Then:
\begin{multline}
\chi\circ \wp(z)=(\psi^\al)^{j+1}(\chi_k(\wp^{-j-1}(\wp(z)))\\
=\psi^\al((\psi^\al)^j(\chi_k(\wp^{-j}(z)))=\psi^\al(\chi(z))
\end{multline}
which concludes the proof.
\end{proof}

It remains to prove that we can construct  $\psi$  such that it satisfies conditions (C1), (C2), (C3) and (C4) of definition \ref{def:modelplan}. Conditions (C1) and (C2) are obviously realized. To get conditions (C3) and (C4), we have to be more precise about the choice of the functions $\xi_k$ defined on the domains $\RR_k$.

\begin{proof}[Proof of lemma II]
For $u\in [0,\frac{1}{p}]$, we set $\de_u:=\{\frac{1}{p}+u\}\times [0,1]\subset \RR_1$.
Using remark  \ref{sig_k,k+1borneInf}, we see there exists $u_0\geq \us$ such that $\psi([0,1], \de_{u_0})\subset \overset{\circ}{\RR_1}$.
Let $u_1\in\, ]u_0, \frac{1}{p}[$ be such that $\psi([0,1], \de_{u_0})\subset [\frac{1}{p}+u_0,\frac{1}{p}+u_1]\times [0,1]$.
Set $\be=\frac{1}{p}-2\us$.

As for $\xi_1$, we choose a  continuous and piecewise $C^1$ function on $\RR_1=[\frac{1}{p}+u_0,\frac{1}{p}+u_1]\times[0,1]$, constant and equal to 
\begin{equation}\label{eq:maj}
M:=\max(\be/2,\max_{\substack{1\leq k\leq p\\r\in [0,1]}}\lam_k(ar)\sqrt{(u^*)^2+\mu_k(ar)})
\end{equation}
over $[\frac{1}{p}+u_0,\frac{1}{p}+u_1]\times [0,1]$, and we choose the values  of $\xi_1$ for $(\th,r)\notin [\frac{1}{p}+u_0,\frac{1}{p}+u_1]\times [0,1]$ in order  to satisfy the relation $(**)$ for each fixed $r$, which is possible since $\sig_{1,2}(r)$ is bounded below by $2$.
For $k\geq 2$, let us choose
\[
\xi_k(\th,r)=\frac{\be}{\sig_{k,k+1}(r)}.
\]
 We moreover require
that 
\[
\xi_1(\th,r)\leq \xi_1(\th,r')
\]
if $r\leq r'$ and $\th\in[\frac{1}{p}+\us,\frac{2}{p}-\us]$. Such a choice is obviously possible since the function $\sig_{1,2}$ is decreasing.
Let us check that conditions (C3) and (C4) are realized.
\vspace{0.2cm}

$\bullet$ We begin with the tameness condition (C4). 
Set $\KK=\psi([0,1],\de_{u_0})$ and fix  two points  $z=(\th,r)$ and $z'=(\th',r)$ in $\KK$, on the same orbit, with lifts $(x,r)$ and $(x',r)$ in the universal covering $\R\times[0,1]$.
We write as usual $\ti\psi$ for the lifted flow.
There exists a unique $t_0\in\, ]0,1]$ such that $(x',r)=\ti\psi^{t_0}(x,r)$.  
Now, setting $(x(t),r)$ and $(x'(t),r)$ for $\ti\psi^t(x,r)$ and $\ti\psi^t(x',r)$, the separation function is defined by $E_{z,z'}(t)=x'(t)-x(t)$.

By construction, for any $(\th, r)\in \KK$, one has
\begin{itemize}
\item  $M \geq \max X_k(\th', r))$ for any $(\th',r)\in \OO_k$ and any $k\in \{1,\ldots,p\}$,
\item $M \geq \max \ha{X}_k(\th', r))$ for any $(\th',r)\in \RR_k$ and any $k\in \{1,\ldots,p\}$.
\end{itemize}
 Therefore,
\begin{equation}\label{septime}
E_{z,z'}(t)\leq \int_t^{t+t_0}Mds=t_0 M,
\end{equation}
Obviously, the maximum of $E_{z,z'}$ is achieved when $t=0$ and the tameness condition is proved for the fundamental domain $\KK$.

\vskip2mm

$\bullet$ The torsion condition (C3) is easy. It suffices to check that (C3) is  satisfied in each domain $\OO_k$ and $\RR_k$, that is, one has to verify that for $r'\geq r$ in $[0,1]$,
\begin{enumerate}
\item for all $\th$ such that $(\th,r)\in \OO_k$,  $X_k(\th,r')>X_k(\th,r)$,
\item for all $\th$ such that $(\th,r)\in \RR_k$, $\ha{X}_k(\th,r)>\ha{X}_k(\th',r)$.
\end{enumerate}
The first point is an immediate consequence of the fact that $\mu_k$ is an increasing function.
The second point is an immediate consequence of the fact that $\sig_{k,k+1}$ is decreasing.
\end{proof}

\subsubsection{ Proof of Theorem 3.}
Consider the $p$-model $\al\otimes\psi$ on $\jA$. 
Recall that for any $z\in \DD_a\setm S$, there exists a unique pair $(t_z^-,t_z^+)\in \R^-\times\R^+$, such that $\phi_H(t_z^-,z)\in S, \, \phi_H(t_z^+,z)\in S$ and $\phi_H(]t_z^-,t_z^+[,z)\cap S=\emptyset$.  Moreover if $z\in f\inv(\rho)$, $t_z^+-t_z^-=\tau(\rho)$.

Consider the map 
$\wti{\chi} : \ov{\DD_a}  \rit \jA$ defined by 
\begin{itemize}
\item $\wti{\chi}(z)=(0,\chi(z))$ if $z\in S$
\item  $\wti{\chi}(z)=\al\otimes\psi(-t_z^-,\wti{\chi}(\phi_H(t_z^-,z)))$
\end{itemize}

\begin{proof}[Proof of Theorem 3] We will prove that $\wti{\chi}$ is a compatible homeomorphism that conjugates $\phi_H$ and $\al\otimes \psi$.
We begin by checking the continuity of $\wti{\chi}$. Obviously, $\wti{\chi}$ is continuous on $\DD_a\setm S$. Let us check the continuity in $S$. 
Let $\eps=\frac{1}{3}\Min_{\rho\in [0,a]}\tau(\rho)$. Let $V^+:=\{z\in \DD_a\,|\, t_z^+\in [0,\eps]\}$ and $V^+:=\{z\in \DD_a\,|\, t_z^-\in [-\eps,0]\}$.
Then $V^+\cup V^-$ is a neighborhood of $S$ and $V^-\cap V^+=S$. Fix $z_0\in S\cap f\inv(\rho)$. One has to check that
\[
\lim_{\substack{z\rit z_0\\z\in V^+}}\wti\chi(z)=\wti\chi(z_0)= \lim_{\substack{z\rit z_0\\z\in V^-}}\wti\chi(z).
\]
Now, if $z\rit z_0$ and $z\in V^+$, then $t_z^+\rit 0$ and $t_z^-\rit -\tau(\rho)$. Therefore
\begin{equation*}
\begin{aligned}
\lim_{\substack{z\rit z_0\\z\in V^+}}\wti\chi(z) &=\lim_{\substack{z\rit z_0\\z\in V^+}}\al\otimes\psi(-t_z^-,\wti\chi(\phi_H(t_z^-,z))) &\\
&=\al\otimes\psi\left(-\tau(\rho),\left(0,\chi(\phi_H(-\tau(\rho),z_0))\right)\right) &\\
& = \al\otimes\psi(-\tau(\rho),(0,\chi\circ \wp\inv(z_0)))\\
&=(0,\psi^\al\circ\chi\circ \wp\inv(z_0)) & \\
&  = (0,\chi\circ\wp\circ \wp\inv(z_0)) \\
&=(0,\chi(z_0)).
\end{aligned}
\end{equation*}
On the other hand, if $z\rit z_0$ and $z\in V^-$, then $t_z^-\rit 0$. Therefore
\[
\lim_{\substack{z\rit z_0\\z\in V^-}}\wti\chi(z)=\al\otimes\psi(0,\wti\chi(\phi_H(0,z_0))=\al\otimes\psi(0,(0,\chi(z_0)))=(0,\chi(z_0)).
\]
By construction, $\wti{\chi}$ is a homeomorphism. It remains to check that it conjugates $\al\otimes \psi$ and $\phi_H$.
Fix $z\in \DD_a\cap f\inv(\{\rho\})$ and $t\in \R$. Let $m\in \Z$ and $s\in [0,\tau(\rho)[$ such that $t=m\tau(\rho)+ s$. Let $z_0:=\phi_H(t_z^-,z)\in S$.
Then
\begin{equation*}
\begin{aligned}
\wti{\chi}\circ\phi_H(t,z) &=\wti{\chi}\circ\phi_H(s-t_z^-,\wp^m(z_0))\\
 &=\al\otimes\psi(s-t_z^-,(0,(\psi^\al)^m\circ\chi(z_0)))\\
 &=\al\otimes\psi(s-t_z^-+m\tau(\rho),(0,\chi(z_0)))\\
 &=\al\otimes\psi(s+m\tau(\rho),(\al\otimes\psi(-t_z^-,\wti\chi(\phi_H(t_z^-,z))))\\
 & =\al\otimes\psi(t, \wti\chi(z)),
\end{aligned}
\end{equation*}
which concludes the proof.
\end{proof}

\subsection{The polynomial entropy of a $p$-model system}

This section is devoted to the proof of proposition \ref{hpolpmodel} in the precise setting we have now at our disposal: namely, the $p$-model
system at hand will be that we constructed in the previous sections. The main remark is that we will have the possibility to choose the minimal 
period of the $p$-model by simply reducing the parameter $a$ of our partial neighborhood.
We first emphasize the following remarks. 

\begin{rem}\label{periode}
\noindent 1) For $r\in\,\,]0,1]$, we denote by $T(r)$  the period of the orbit $\T\times \{r\}$. By
condition (C1)  
\[
T(r)\sim_{r=0}-\sum_{\ell=1}^{p-1}\frac{\lam_\ell(r)}{\ln(\mu_\ell(r))}\sim_{r=0}-\sum_{\ell=1}^{p-1} \frac{\lam_\ell(r)}{\ln r}, 
\]
the last equivalent coming from $\mu_k(0)'\neq 0$.

\noindent 2) 
Let $\be$ be such that $[\kp-\be, \kp+\be]\subset \OO_k$, for any $1\leq k\leq p$, and set
\[
\tau_k(r,\be):=\int_{\kp+\be}^{\kp-\be} \frac{d\th}{\lam_k(r)\sqrt{(\kp-\th)^2+\mu_k(r)}}.
\]
One has:
\[
\frac{T(r)}{\tau_k(r,\be)}\sim_{r=0}\sum_{\ell=1}^{p-1}\frac{\lam_\ell(0)}{\lam_k(0)}.
\]
\noindent 3) Due to the torsion condition, $r\ma T(r)$ is a strictly decreasing function from $]0,1]$ to $[q^*,+\infty[$, where $q^*$ is the period of the motion $\psi$ on
$\T\times\{1\}$.
\end{rem}

\begin{nota}
For the sake of simplicity, in the whole proof of proposition \ref{hpolpmodel}, we write $\phi^t$ instead of $(\al\otimes\psi)^t$.
\end{nota}

\newcommand{\hd}{\hat{d}}

\begin{proof}[Proof of proposition \ref{hpolpmodel}]
We denote by $a$ an element $(\th,r)\in \ha{\jA}$. For $\vp\in \T$ and $k\in \N$, we set $\vp_r(k):=\vp+k\al(r)$. Therefore
\[
\phi^k(\vp,a)=(\vp_r(k),\psi^k(a)).
\]
We denote by $\de$ the natural quotient distance on $\T$ and by $\hat{d}$ the product metric on $\T\times [0,1]$. We denote by $d:= \de\times \hd$, the product metric on $\jA$.
For $k\in \N$, we denote by $d_k$ the dynamical distances in $\T^2\times [0,1]$ associated with the motion $\phi$ and by $\hat{d}_k$ the distances in $\T\times[0,1]$ associated with $\psi$.

Remark that for $\vp$ in $\T$ and $a=(\th,r),a'=(\th',r')$ in $A\subset \T\times[0,1]$, one has:
\begin{align*}
d_N^\phi((\vp,a),(\vp',a'))  &=\max_{0\leq k\leq N}d(\phi^k(\vp,a),\phi^k(\vp',a'))\\
&= \max_{0\leq k\leq N}\max\bigg(\de(\vp_r(k),\vp_{r'}'(k)),\hd(\psi^k(a),\psi^k(a'))\bigg)\\
&= \max\left( \max_{0\leq k\leq N}\left(\de(\vp_r(k),\vp_{r'}'(k)\right),\max_{0\leq k\leq N}\hd(\psi^k(a),\psi^k(a'))\right)\\
& = \max\left( \max_{0\leq k\leq N}\left(\de(\vp_r(k),\vp_{r'}'(k)\right),\hd_N(a,a')\right)\qquad\qquad (\star)
\end{align*}

Let us introduce some notation.
\vskip1mm\noindent
-- Given $r\in\,]0,1]$, we set $C^r:=\T\times \{r\}\subset \ha{\jA}$.\index{$C^r$}
\vskip1mm\noindent
-- According to remark \ref{periode} one can  label the orbits $C^r$ by their period: we write $C_q$ the orbit with period 
$q$ for $q\in [q^*,+\infty[$. So $C^r=C_{T(r)}$.\index{$C_q$}
We write $C_{\infty}$ for the boundary $\T\times\{0\}$. 
\vskip1mm\noindent
-- Given two periods $q'\leq q\leq+\infty$, we denote by $\ha{S}_{q,q'}$ the annulus $\subset \ha{\jA}$
bounded  by the curves $C_q$ and $C_{q'}$.\index{$\ha{S}_{q,q'}$}
\vskip1mm\noindent
 -- When $a$ and $b$ are two points on the same curve $C_q$, we denote by
$[a,b]$ the set of all points of $C_q$ located between  $a$ and $b$, relatively to the direct orientation of $C_q$.
\vskip1mm\noindent
-- We denote respectively by $\phi$ and by $\psi$ the time-one  maps of the flows $(\phi^t)_t$ and $(\psi^t)_t$.
\vskip5mm

\paraga \emph{Proof of $\hp(\phi)\geq 2$.}
Recall  that for $\eps>0$, a subset $A$ of a compact metric space  $(X,d)$ is said to be \textit{$(N,\eps)$-separated} relatively to a continuous map $f:X\rit X$ if, for any $a$ and $b$ in $A$,  $\sup_{0\leq k\leq N}d(f^k(a),f^k(b))\geq \eps$. 
If $S_N^f(\eps)$ is the maximal cardinal for a $(N,\eps)$-separated set, one has the following inequalities
\[
S_N^f(2\eps)\leq G_N^f(\eps)\leq S_N^f(\eps).
\]
Therefore, $\hp(f) =\lim_{\eps\rit 0}\limsup_{n\rit\infty}\frac{\Log S_n^f(\eps)}{\Log n}$.
Given $\eps>0$, we want to find, for $N$ large enough, a $(N,\eps)$-separated set (relatively to $\phi$) with cardinal $\geq c_0 N^2$ for a constant $c_0>0$.

It suffices to find  a $(N,\eps)$-separated in $\T\times [0,1]$ with cardinal $\geq c_0 N^2$ relatively to $\psi$. 
Indeed assume we are given such a set $A(N,\eps)\in \T\times [0,1]$ .
Then by $(\star)$, since $\hd_N(a,a')\geq \eps$, one has $d_N((\vp,a),(\vp,a'))\geq \eps$.

Fix a vertical segment $I_{\th_0}:=\{\th=\th_0\}\subset \T\times[0,1]$. For $q\in [q^*,+\infty]$, we set $a_q:=I_{\th_0}\cap C_q$. 
We choose $\th_0$ such that $a_\infty$ is not a singular point of $V$, that is $a_\infty\neq (\kp,0)$, so 
$\psi(a_\infty)\neq a_\infty$ and $\psi\inv(a_\infty)\neq a_\infty$. 

\begin{rem}\label{intervalCinfini}
Due to the torsion condition, for each $q\geq 3$ the projection on $C_\infty$ of the interval $[\psi(a_q),\psi^{[q/2]}(a_q)]$ 
is contained in $[\psi(a_\infty),\psi\inv(a_\infty)]$. 
\end{rem}
Assume the two (redondant) conditions
\[
\eps<\min\left(\hd(a_\infty, \psi(a_\infty)), \hd(a_\infty, \psi\inv(a_\infty))\right)\hspace{4cm} (1)
\]
and that 
\[
\eps <\min\left(\hd(a_\infty, \psi\inv(a_\infty)),\hd(\psi(a_{\infty}),\psi^2(a_{\infty})\right).\hspace{4cm} (2)
\]
In the following we assume that $q\geq 3$.
\vskip2mm

\noindent\textbf{Step 1}: \emph{If $N\geq q$, $C_q$ contains an $(N,\eps)$-separated set with cardinal $[q/2]$. } Fix  $N\geq q$.
For $k\geq0$ we set $a^{(k)}=\psi^{-k}(a_q)$.
\begin{proof} Then, by $(1)$ and remark \ref{intervalCinfini} and since
$\psi^{k'-k}(a_q)\in[\psi(a_q),\psi^{[q/2]}(a_q)]$, one has
\[
\hd_N\big(a^{(k)},a^{(k')}\big)\geq \hd\big(\psi^{k'}(a^{(k)}),\psi^{k'}(a^{(k')})\big)=\hd\big(\psi^{k'-k}(a_q),a_q\big)> \eps,
\]
forall $0\leq k<k'\leq [q/2]$.
Therefore, the set $\{a^{(k)}\mid 1\leq k\leq [q/2]\}$ is $(N,\eps)$-separated
\end{proof}

\noindent
\textbf{Step 2} \emph{If $N\geq 18$ and $(q,q')\in [\frac{N}{3},\frac{N}{2}]^2$ with $3\leq q-q'\leq q-3$, the pairs of points $(a,a')\in C_q\times C_{q'}$ are  $(N,\eps)$-separated.}
\begin{proof}Let us  introduce the domains: 
\[
I_q=[a_q,\psi(a_q)[\ \subset C_q,\qquad J_{q'}=[\psi\inv(a_{q'}),\psi^2(a_{q'})[\ \subset C_{q'}.
\]
Thanks to the torsion condition,  by $(2)$, the distance between $I_q$ and the complement  $C_{q'}\setm J_{q'}$ is
larger than $\eps$, for each pair $(q,q')$ in $[q^*,+\infty[$.

Now assume that $q$ and $q'$ are contained in the interval $[N/3,N/2]$ and satisfy $q-q'\geq 3$.
Consider two points $a\in C_q$ and $a'\in C_{q'}$. 
There exists a unique integer $n_0\in \{0,\ldots,q-1\}$ such that 
$\psi^{n_0}(a)\in I_q$. Thus:
\vskip1mm
\noindent (i)
if $\psi^{n_0}(a')\in C_{q'}\setm J_{q'}$, then $d_N(a,a')\geq d_{n_0}(a,a')>\eps$. 
\vskip1mm\noindent
\noindent (ii)
if $\psi^{n_0}(a')\in J_{q'}$, observe that $\psi^{q-q'}(J_{q'})\cap J_{q'}=\emptyset$. Indeed,
\[
\psi^{q+n_0}(a')=\psi^{q} (\psi^{n_0}(a'))=\psi^{q-q'} (\psi^{n_0}(a')),
\]
with $q'\geq 6$ and $3\leq q-q'\leq q-3$ by assumption on $N$ and $q,q'$. So $\psi^{q+n_0}(a')\notin J_q$ and,
by periodicity, $\psi^{q+n_0}(a)\in I_q$.

Therefore, using $(2)$:
$
\hd_N(a,a')\geq \hd(\psi^{q+n_0}(a),\psi^{q+n_0}(a'))>\eps.
$
\end{proof}

\noindent\textbf{Step 3}: \emph{If $N\geq 18$, $S_{N/3,N/2}$ contains a $(N,\eps)$-separated set with cardinal $\geq N^2/108$.}
\begin{proof} In the  interval $[N/3,N/2]$, there exist at least $[N/18]$ distinct integers   $(q_i)$ with $q_j-q_i\geq 3$ if $i\neq j$. On each curve $C_{q_i}$, one can find an $(N,\eps)$--separated
subset with $[q_i/2]\geq [N/6]$ elements by step 1, and the union of all these subsets is still $(N,\eps)$--separated by step 2. 
Therefore the strip limited by the curves $C_{N/3}$ and $C_{N/2}$ contains a $(N,\eps)$--separated subset
$A(N,\eps)$ with more than $N^2/108$ elements.
\end{proof}

\noindent\textbf{Conclusion}: Fix $\vp\in \T$. The set $\{(\vp,a)\,|\, a\in A(N,\eps)\}$ is $(N,\eps)$-separated with cardinal $\geq N^2/108$. Therefore, for $N>18$, $S_N^\phi(\eps)\geq N^2/108$ and $\hp(\phi)\geq 2$.
\vspace{0.5cm}

\paraga \emph{Proof of $\hp(\phi)\leq 2$.}
If $f$ is a continuous map on a compact set $X$, we 
denote by $D_n^f(\eps)$ the smallest number of sets $X_i$ with $d_n^f$-diameter $\eps$ whose union covers $X$.
Then, one has 
\[
D_n^f(2\eps)\leq G_n^f(\eps) \leq D_n^f(\eps).
\]
Therefore, $\hp(f) =\lim_{\eps\rit 0}\limsup_{n\rit\infty}\frac{\Log D_n^f(\eps)}{\Log n}$.

The main idea of the construction will be to take advantage of the explicit coverings constructed in \cite{Mar-09} for planar $p$-models
and to prove that in the product system on $\jA$, the domains involved in this covering are so small that they induce a negligible distorsion
in the $\varphi$ variable. So we will be able to construct a covering for the present $p$-model simply by taking the product of the
domains of the planar $p$-model with small enough intervals in the $\varphi$-direction.

If $A$ is a subset of $\jA$ invariant by $\phi$,  for $n\geq 1$, we denote by $D_n^\phi(A,\eps)$ the minimal cardinal of a covering of  $A$ by subsets with $d_N$-diameter $\eps$. We define in the same way, for  $\ha{A}\subset \hjA$ invariant by $\psi$, the number $D_n^\psi(\ha{A},\eps)$.
Given $\eps>0$, we want to get, for $N$ large enough, a majoration of the form $D_N^\phi(\jA,\eps)\leq c_0 N^2+c_1 N+c_2$.
To do this, we discriminate between the behavior of the dynamics near $r=0$ and the behavior of the dynamics near $r=1$. We split $\ha{\jA}$ into two $N$-depending sub-annuli $\ha{\jA}_N$ and $\ha{\jA}_N^*$ and we estimate $D_N^\phi(\jA_N,\eps)$ and $D_n^\phi(\jA_N^*,\eps)$, where $\jA_N:=\T\times \hjA_N$ and $\jA_N^*:=\T\times \hjA_N^*$. Fix $\eps>0$.

\vskip2mm
\noindent\textbf{Step 1: Construction of the sub-annuli $\ha{\jA}_N$ and $\ha{\jA}_n^*$}.
The choice of the cutoff is based on the following lemma.

\begin{lem}\label{lem:choice} 
For $k\in\{1,\ldots,p\}$, let  $B_k(\eps)$ be the ``block'' of $\hjA$ limited by the  vertical
segments $\De^+_k$ and $\De^-_k$ of equations $\th=k/p-\eps/2$ and $\th=k/p+\eps/2$ respectively.
There exists a constant $\ka$ and an integer $N_0$  (both depending on $\eps$) such that if $N\geq N_0$,  for each index $k\in\{0,\ldots,p-1\}$:
\[
\psi^n(\De^-_k(\ka N))\subset B_k(\eps),\qquad \forall n\in\{0,\ldots,N\},
\]
where we write $\De^-_k(q)$ for the intersection of the left vertical $\De^-_k$ of $B_k(\eps)$ with the annulus $\ha{S}_{\infty,q}$.
\end{lem}

\begin{proof}
With the notation of remark \ref{periode} (2), we set 
\[
\ka:=\left[2\max \sum_{\ell=1}^{p-1}\frac{\lam_\ell(0)}{\lam_k(0)}\right]+1
\]
Then, by remark \ref{periode}, there exists $r_0>0$ such that for $r<r_0$, $\tau_k(r,\eps/2)>\frac{1}{\ka}T(r)$, that is,
\[
\psi^{\frac{1}{\ka}T(r)}(\De_k^-(T(r))\in B_k(\eps).
\]
The lemma is proved with $N_0:=\lceil T(r_0)\rceil$.
\end{proof}

\newcommand{\hS}{\ha{S}}

For $q\in [q^*,+\infty[$, we write $r(q):=T\inv(q)\in\, ]0,1]$. 
By remark \ref{periode} (3), there exists $N_1\in \N^*$ such  that if $q>N_1$, $r_q\leq \eps$.

We choose $N\geq \Max(N_0, N_1)$ and we set: $\hjA_{N}=\hS_{\infty,\ka N}$,\, $\hjA^*_{N}=\hjA\setm \hjA_{N}$,\, $\jA_N:=\T\times \hjA_N$ and 
$\jA_N^*:=\T\times \hjA_N^*$.
\vspace{0.2cm}
 
We will use twice the following easy remark.

\begin{rem}\label{periode2} 
By remark \ref{periode} (1), for $q^*$ large enough, there exists $\bar c\in \R$ such that
\[
r'(q)\geq e^{-\bar c q}.
\]

%
\end{rem}


\vskip2mm
\noindent\textbf{Step 2: Covering of $\jA_{N}$ }.
We begin by constructing a covering of $\hjA_N$.
For $k\in\{1,\ldots,p-1\}$ and $r\in [0,1]$, we write $a_k(r):=(\frac{k}{p}+\frac{\eps}{2},r)=\De_k^+\cap C^r$.
We set
\[
\nu_k:=\Min\{\ell\in \N^*\,|\,\psi^{\ell}(a_k(0))\in B_{k+1}(\eps)\}\quad\rm{	and}\quad \nu=\Max_k\nu_k.
\]
By the torsion property, and since $\tau_{k+1}(r,\eps/2)\rit\infty$ when $r\rit 0$,  there exists $r^*$ such that if $r\in [0,r^*]$, 
$\psi^\nu(a_k(r))\in B_{k+1}(\eps)$, for all $k\in \{1,\cdots,p\}$.
We set $N_2:=\frac{1}{\ka}\lceil T(r^*)\rceil$. Then for $N\geq N_2$, 
\begin{equation}\tag{**}
\psi^{\nu}(\De_k^+(\ka N))\subset B_{k+1}(\eps).
\end{equation}

We set $\BB=\cup_{1\leq k\leq p}B_k(\eps)$.
By compactness, there exists a finite covering $B_1,\ldots, B_{i^*}$  of $\hjA_N\setm \BB$ with $\hd_\nu$--diameter  
$\leq\eps$. Moreover, one can obviously assume that each $B_i$ is contained in some connected component of
$\hjA_N\setm \BB$. 

We claim that, if $N>\Max(N_0,N_1, N_2)$, for any $n$ in $\{\nu, \ldots, N\}$,  $\psi^n(B_i)$ is contained in some $B_k(\eps)$. Indeed, assume that $B_i$ is contained in the zone limited by the curves $\De_k^+(\ka N)$ and 
$\De_{k+1}^-(\ka N)$  (according to the direct
orientation on $\T$). Then the iterate $\psi^n(B_i)$ is contained in the region limited by  $\psi^\nu(\De_k^+(\ka N))$ and $\psi^N(\De_{k+1}^+(\ka N))$. Both of these boundaries are contained in $B_{k+1}(\eps)$, the first one by $(**)$ and the second one by lemma \ref{lem:choice}.

By assumption on $N$, the $d$--diameter of $B_{k+1}(\eps)$ is $\eps$. Therefore,
\[
\diam_N(B_i)\leq \eps 
\]
where we denote by $\diam_N$ the diameter associated with the distance $\hd_N$.


Now, we can assume that $\eps$ is small enough and $N$ is large enough so that the intersection $\psi(\De_{k}^-(\ka N))\cap B_{k+1}(\eps)$ is empty.
Consider the regions $\UU_k$ in $\hjA_{N}$ bounded by $\De_k^+(\ka N)$ and $\psi(\De_{k}^+(\ka N))$.

The nonempty intersections $\UU_k\cap B_i$ form a finite covering $U_1,\ldots, U_{i^{**}}$ of the union $\cup_{1\leq k\leq p}\UU_k$, with cardinal $i^{**}\leq i^*$.
For $1\leq i\leq i^{**}$, one has:
\[
\diam_N(U_i)\leq \eps.
\]  
Let $\VV_k$ be the region bounded by $\psi^{-N}(\De_k^+(\ka N))$ and $\De_k^+(\ka N)$ (relatively to the direct orientation of $\T$). 
By the same arguments as in the beginning, one sees that $\VV_k\subset B_k(\eps)$. 
Moreover the inverse images
\[
B_{n,i}=\psi^{-n}(U_i),\qquad 1\leq n\leq N, \ 1\leq i\leq i^{**},
\]
form a covering of the union $\VV=\cup_{k\in\Z_p} \VV_k$ which is contained in $B_k(\eps))$. 
By construction, each of the $B_{n,i}$ satisfies $\diam_NB_{n,i}\leq \eps$.

Finally,  observe that for each $k$, the complement $B_k(\eps)\setm \VV$ satisfies 
\[
\psi^n(B_k(\eps)\setm \VV)\subset B_k(\eps)
\]
for $0\leq n\leq N$, and therefore $\diam_N(B_k(\eps)\setm \VV)\leq\eps$. 
Hence, the subsets 
\[
(B_i)_{1\leq i\leq i^*},
\quad(B_{n,i})_{1\leq n\leq N,\ 1\leq i\leq i^{**}},\quad
(B_k(\eps)\setm \VV)_{1\leq k\leq p},
\]
form a covering $\CC(N,\eps)$ of $\hjA_{N}$ with subsets of $\hd_N$--diameter $\leq\eps$ and 
\[
\Card \CC(N,\eps)\leq i^*+Ni^{**}+p.
\]

We use $(\star)$ to construct a covering of $\jA_N$ with $d_N$-balls of radius $\eps$.
Fix an element $B$ of $\CC(N,\eps)$. Let $a=(\th,r), a'=(\th',r')$ be two elements of $B$ and let $(\vp,\vp')\in \T^2$,  
By $(\star)$, one has 
\[
d_N((\vp,a),(\vp',a'))=\max\left( \max_{0\leq k\leq n}\left(\de(\vp_r(k),\vp_{r'}'(k)\right),\eps\right).
\]
Now, for any $1\leq k\leq N$, 
\[
\de(\vp_r(k),\vp'_{r'}(k))\leq \de(\vp,\vp')+k|\al(r)-\al(r')|\leq  \de(\vp,\vp')+N\max_{[0,1]}\al'(r)|r-r'|.
\]
By remark \ref{periode2}, if $q^*$ is large enough (which is always possible, reducing if necessary the width of the partial neighborhood of the
initial polycycle), there exists  $N_3$ such that for all $N\geq N_3$, 
\[
N\max_{[0,1]}|\al'(r)||r-r'|\leq N\max_{[0,1]}|\al'(r)|\ r(\ka N)\leq \eps/2.
\]

Therefore, if $I$ is a subset of $\T$ with $\de$-diameter $\eps/2$, the product $I\times B$ with $B\in \CC(N,\eps)$ is a subset with $d_N$-radius $\eps$ as soon as $N\geq \max(N_0,N_1,N_2,N_3)$.
This yields  a cover of $\jA_N$ with subsets with $d_N$-diameter $\eps$ of cardinal $\leq  \frac{2}{\eps}(i^*+Ni^{**}+p)$.
 
\vskip4mm
\noindent\textbf{Step 3: Covering of $\jA_{N}^*$ }.
Set $k^*=[\ka N/q^*]-1$. For $1\leq k\leq k^*$, we set $\ha{{\bf S}}_k:=\ha{S}_{\frac{\ka N}{k},\frac{\ka N}{(k+1)}}$.
 Clearly, the family $(\ha{{\bf S}}_k)$ covers $\hjA_{N}^*$. 
Assume we are given a minimal covering $\CC_k(N,\eps)$ of $\ha{{\bf S}}_k$ with subset of $\hd_N$-diameter smaller than $\eps$.
To form a covering of the complete annulus $\jA$, we will see that it is enough to take the product of the elements of $\CC_k(N,\eps)$ 
with small enough intervals in the $\varphi$ direction. Let us first construct the
covering $\Card \CC_k(N,\eps)$, since the form of its elements will play a crucial role.

 \begin{lem}\label{lem:fatten}
Consider an integer $m\geq q^*$, and fix $\eps>0$. There exists positive constants 
$c_1$  and $c_2$, depending only on $\eps$, 
such that if  the pair $(q,q')\in [q^*,m]^2$ satisfy
\[
0\leq q'-q \leq \dfrac{c_1\,\eps}{[m/q]},
\]
then the sub-annulus $\ha{S}_{q,q'}$ satisfies
\[
D_m^\psi(\ha{S}_{q,q'},\eps)\leq c_2\, q.
\]
\end{lem}
 
\begin{proof}  We  first construct a covering of a single curve $C_q$, then we fatten it a little bit to get a covering of a thin strip $S_{q,q'}$. We will use condition (C3). Recall that $\KK$ is the fundamental domain for the tameness condition.

Fix $q\in [q^*,m]$.
Let $\lam$ be the Lipschitz constant of $\Psi$ on the compact set  $[-1,1]\times\KK$. 
Let $I_q$ be the interval $C_q\cap \KK$. Consider two points $a\leq a'$ contained in 
$I_q$. 
Then by the tameness property the maximum $\mu$ of the separation function $E_{a,a'}$ is achieved for $t$ such that $\psi_t(a)$ and $\psi_t(a')$ are located inside $I_q$. 
Therefore there exists $t_0\in[-1,1]$ and $n\in \N$ such that $t=t_0+nq$. 
As a consequence, $\mu\leq \lam\, d(a,a')$. 

Hence, for all $k\in \N$,  $d_k(a,a')\leq\lam\, d(a,a')$.

For $q\geq q^*$ we set $j_q^*:=[\frac{\diam I_q}{\eps/(2\lam)}]+1$ and we cover $I_q$ by consecutive subintervals $J_1,\ldots,J_{j^*_q}$ of $d$--diameter $\eps/(2\lam)$.
 As $I_q,\psi(I_q),\ldots,\psi^{[q]}(I_q)$ is a covering of $C_q$, one sees that the intervals $I_{ij}=\psi^i(J_j)$, ${0\leq i\leq [q], 1\leq j\leq j^*_q}$ form a covering of $C_q$ by subsets of $d_k$--diameter $\leq \eps/2$, for {\em each} integer $k$.
 Indeed, if $a,a'$ lie in $I_{ij}\subset I_q$, $d_k(a,a')\leq \lam \eps/(2\lam)$.

Due to  the torsion property, for $q\leq q^*$, $j^*_q\geq j^*_{q^*}$. We set $c_2=2j^*_{q^*}$. Then, for each $q\in[q^*,m[$,  each orbit $C_q$ admits a covering by at most $c_2\,q$ subsets whose  $\hd_k$--diameter is smaller than $\eps/2$, for any positive integer $k$.
\vskip2mm

Fix now an integer $m\geq q^*$. 
Given the initial period $q$, we want to find a period $q'\leq q$ such that for any pair of points $a\in C_q$ and $a'\in C_{q'}$ {\em with the same abscissa 
$\th$}, the (maximal) difference of the abscissas of any pair of iterates $\psi^n(a)$ and $\psi^n(a')$, $n\in\{0,\ldots,m\}$, is at most $\eps/2$. 

Assume this is done and consider again the covering of $C_q$ by the intervals $I_{ij}$. 
Fix such an interval $I_{ij}:=[\th^-,\th^+]$ and let $R_{ij}$  be the rectangle  limited by the curves $C_q$ and $C_{q'}$ and the vertical lines $\th=\th^-$ and $\th=\th^+$.
Fix a lift to the universal covering $\R\times[0,1]$ and consider the associated the lifted flow $\wti\psi_t$. For $a\in R_{ij}$ and $t>0$, we set $(x(t),r)=\wti\psi_t(\ti a)$.
We set $a^-=(\th^-,r_q)$ and $a^+=(\th^+,r_{q'})$. By torsion property, one has  (with obvious notation):
\[
x^-(t)\leq x_1(t)<x_2(t) \leq x^+(t),\quad \forall\, a_1,a_2\in R_{ij}, \quad \forall t>0.
\]
Therefore, $\hd_m(a_1,a_2)\leq \hd_m(a^-,a^+)$ and the rectangles $R_{ij}$ have  $\hd_m$-diameter less than $\eps$. They form  a covering of the strip $\ha{S}_{q,q'}$  with at most $c_2\,q$ elements.

So it remains to choose $q'$ close enough to $q$.
Fix two points $a,a'$ located in the same vertical and denote by $\ti a$ and $\ti a'$ two lifts (located in the same vertical).
As before, we set $\ti\psi_s(\ti a)=(x(s),r)$, $\ti\psi_s(\ti a')=(x'(s),r')$, so  $r'\geq r$ since $q'\leq q$. 
Given $t\geq 0$, we denote by $t'$ be the time   needed for the point $a'$ to reach the vertical through $a(t)$. So $t'$ is defined by the equality 
\[
x'(t')=x(t).
\]
We set 
$
\De(a,t)=t-t'
$
so, by the torsion property,  $D(a,t)\geq0$. One easily checks that
\[
\De(a,t_1+t_2)=\De(a,t_1)+\De(\psi_{t_1}(a),t_2),\quad  \De(a,kt)=\sum_{\ell=0}^k\De(\psi_\ell(a),t).
\]
The first equality shows that $t\ma \De(a,.)$ is an increasing funtion. When $a\in C_q$, the second equality yields 
$\De(a,kq)=k\De(a,q)$.
It is also easy to see that
\[
\De(a,q)= q-q',\quad  \forall a\in C_q.
\]
Now set $\ell':=\max(\ell, \max_{\substack{(\th,r)\in\OO_k\\1\leq k\leq p}}V_k(\th,r))$, where is $\ell$ defined by Condition (C1). 
Then:
\[
0\leq x'(t)-x(t)\leq \ell'\, D(a,t).
\]
For $a\in C_q$, one has:
\[
\De(a,m)\leq  \De(a,([\frac{m}{q}]+1)q)=([\frac{m}{q}]+1)(q-q').
\]
Consequently, for $0\leq n\leq m$,
\[
0\leq x'(n)-x(n)\leq \ell'\, \De(a,n)\leq \ell' \, \De(a,m)\leq \ell' \,([\frac{m}{q}]+1)(q-q'). 
\]
which proves our statement for $c_1=\frac{1}{\ell'}$.
\end{proof}


We want to apply lemma \ref{lem:fatten} with $m=\ka N$ and $q\in\,]\frac{\ka N}{(k+1)},\frac{\ka N}{k}]$ for $1\leq k\leq k^*$.
Fix $k$ and assume that  $q\in\,]\frac{\ka N}{(k+1)},\frac{\ka N}{k}]$. Then
\[
\Big[\frac{\ka N}{q}\Big]=k.
\]
Therefore,  by lemma \ref{lem:fatten}, if $q'-q\leq c_1\eps/k$, the strip $\ha{S}_{q,q'}$ satisfies
\[
D_{\ka N}^\psi(\ha{S}_{q,q'},\eps)\leq c_2\, q\leq c_2 \, \frac{\ka N}{k}.
\]
This upper bound is therefore constant on $\ha{{\bf S}}_k$. Now the strip $\ha{{\bf S}}_k$ is covered
by the strips $(\ha{S}_{q_{i+1},q_i})_{0\leq i\leq i^*(k)}$, with
\[
q_i=\frac{\ka N}{k+1}+ i \frac{c_1\eps}{k},\qquad i^*(k)=\Big[\frac{\ka\, N}{c_1\,\eps\,(k+1)}\Big]+1\leq c_3\frac{\ka\, N}{c_1\,\eps\,k}
\]
for $c_3>0$ large enough.
So we can choose $\CC_k(N,\eps)$ such that
\[
\Card\CC_k(N,\eps)\leq D_{N}^\psi(\ha{{\bf S}}_k,\eps)\leq c_2 \, \frac{\ka N}{k}\,i^*(k)\leq  c_\eps\frac{N^2}{k^2},\qquad c_\eps=\frac{c_2c_3\ka^2}{c_1\eps}.
\]

Now let us estimate the maximal width $\Delta r$ of the substrips $(\ha{S}_{q_{i+1},q_i})_{0\leq i\leq i^*(k)}$ in the $r$ variable. Ther width in the $q$ variable is 
$c_1\eps/k$ and their are contained in the strip $\ha{{\bf S}}_k$, whose minimal period is $\ka N/(k+1)$. Therefore, according to the estimate
on $r'(q)$ if $q^*$ is assumed to be large enough (which, as observed above is always possible):
\[
\Delta r \leq  \frac{c_1\eps}{k} \,e^{-\bar c \frac{\ka N}{k+1}}\leq \frac{q^*c_1\eps}{\ka N} e^{-\bar c q^*/2}.
\]
For $q^*$ large enough, this width satisfies
\[
\Delta r \max_{[0,1]}|\al'(r)|\leq \frac{\eps}{2N}.
\]
Therefore $N$ iterations of two points $(\varphi,x)$ and $(\varphi',x')$ with $x$ and $x'$ in the same domain of the covering $\CC_k(N,\eps)$ 
produce a distorsion of at most $\eps/2$ in the $\varphi$-direction. As a consequence, one gets a covering by subsets of $d_N$-diameter
less than $\eps$ by taking the products of the elements of $\CC_k(N,\eps)$ by intervals of uniform length $\eps/2$ in the $\varphi$-direction.

Finally, since the strips ${\bf S}_k$ cover $\jA^*_N$, one gets
\[
D_{N}^\phi(\jA^*_N,\eps)\leq\frac{2}{\eps}\sum_{k=1}^{k^*}D_{N}^\psi(\ha{{\bf S}}_k,\eps)\leq \frac{2}{\eps}\sum_{k=1}^\infty c_\eps\frac{N^2}{k^2}=\al_\eps N^2,
\]
with $\al_\eps=\frac{2}{\eps}c_\eps\zeta(2)$. 
\vspace{2mm}

\noindent  \textbf{Conclusion:}  
Gathering  steps 2 and 3, one has:             
\[
D_N^\phi(\jA,\eps)\leq D_N^\phi(\jA_N^*,\eps)+D_N^\phi(\jA_N,\eps)\leq \ti\al_\eps N^2+\frac{2}{\eps}(i^*+Ni^{**}+p),
\]
for any $N\geq \max(N_0,N_1,N_2,N_3,N_4,N_5)$. 
Therefore $\hp(\phi)\geq 2$.
 \end{proof}

\bibliographystyle{alpha}
\bibliography{bibliohpol}

\end{document}